\documentclass[10pt]{amsart}
\usepackage{geometry}
\usepackage[small]{caption}
\usepackage{amsfonts, amsmath, amssymb, mathrsfs}
\usepackage[foot]{amsaddr}
\usepackage{amsthm, url}
\usepackage[all]{xy}
\usepackage{graphicx, 
epstopdf}
\usepackage{subfig}
\usepackage{color}
\usepackage{bbm}
\usepackage[draft,footnote,nomargin,inline]{fixme}
\usepackage{bm}
\usepackage{epigraph}
\usepackage{amscd}

\usepackage{ifpdf}

\ifpdf
\else
\fi

\makeatletter
\@namedef{subjclassname@2010}{%
 \textup{2010} Mathematics Subject Classification}
\makeatother

\makeatletter
\newif\if@restonecol
\makeatother

\newtheorem{theorem}{Theorem}[section]

\newtheorem*{theoremNN}{Theorem}
  \newtheorem{lemma}[theorem]{Lemma}







 \theoremstyle{definition}
 \newtheorem{definition}[theorem]{Definition} 
  
  \newtheorem{notation}[theorem]{Notation}
\newtheorem{exampleth}[theorem]{Example}
\newenvironment{example}{\begin{exampleth}}{\hfill $\diamond$\\ \end{exampleth}}
\makeatletter
\newenvironment{examplethNN}[1]{\par
  \normalfont \topsep6\p@\@plus6\p@\relax
  \trivlist
  \item[\hskip\labelsep
    \textbf{Example #1}\@addpunct{.}]\ignorespaces
}{%
  \endtrivlist\@endpefalse
}
\makeatother

\newcommand \QQ {\mathbb{Q}} 
\newcommand \ZZ {\mathbb{Z}}
\newcommand \CC {\mathbb{C}}
\newcommand \RR {\mathbb{R}}

\newcommand \pr {\mathbb{P}}

\newcommand \OO {\mathcal{O}}

\newcommand \HH{\operatorname{Hom}}

\newcommand \afc {\mathbb{C}}

\newcommand \Sn {\mathbb{S}}
\newcommand \init {\operatorname{in}}

\newcommand \cT {\mathcal{T}}

\newcommand \cP {\mathcal{P}}

\newcommand \cF {\mathcal{F}}
\newcommand \sat{\operatorname{sat}}
\newcommand \TP {\mathbb{T}}
\newcommand \sign {\operatorname{sign}}

\newcommand{\val}{\operatorname{val}}
\newcommand{\ord}{\operatorname{ord}}

\newcommand{\NF}{\mathscr{N}}

\newcommand {\trop} {\operatorname{trop}}
\newcommand {\CF}{\mathscr{C}}

\title{Implicitization of surfaces via geometric tropicalization}
\author{Mar\'ia Ang\'elica Cueto$^\dagger$}

\address{$^\dagger$ FB12 Institut f\"ur Mathematik,
Goethe-Universit\"at Frankfurt, Robert-Mayer-Str. 6-8, Raum 214,60325
Frankfurt am Main,Germany. Phone: +49 (69)-798-28929, Fax: +49 (69) 798-22302.}
\email{macueto@math.uni-frankfurt.de}

\date{\today}

\begin{document}
\maketitle

\begin{abstract} 
  In this paper we further develop the theory of geometric
  tropicalization due to Hacking, Keel and Tevelev and we describe
  tropical methods for implicitization of surfaces. More precisely, we
  enrich this theory with a combinatorial formula for tropical
  multiplicities of regular points in arbitrary dimension and we prove
  a conjecture of Sturmfels and Tevelev regarding sufficient
  combinatorial conditions to compute tropical varieties via geometric
  tropicalization. Using these two results, we extend previous work of
  Sturmfels, Tevelev and Yu for tropical implicitization of generic
  surfaces, and we provide methods for approaching the non-generic
  cases.
\end{abstract}
\section{Introduction}
\label{SIsec:introduction}

In its ten years of existence, the field of tropical geometry  has
provided new tools to approach questions in algebraic geometry. Among
them, we can include classical elimination and
implicitization problems~\cite{TropDiscr, ElimTheory,
  NPofImplicitEquation,TropImplMixedFibPolytopes}. In the classical
setting, we wish to recover the defining ideal of either the projection of a
subvariety of an algebraic torus or of a parametric variety. In the
tropical setting, we replace the defining ideal by a polyhedral
object, namely, its  tropicalization.  Such methods are known as
\emph{tropical elimination} and \emph{tropical implicitization} and
have been used recently for computations going beyond the power of
classical elimination tools, including multidimensional resultants and
Gr\"obner bases.  Successful applications of tropical implicitization
techniques were presented in \cite{FPSAC2010CL, Mega09}.
 
Tropical geometry is a polyhedral version of classical algebraic
geometry: we replace algebraic varieties over the torus
$\TP^r=(\CC^*)^r$ by weighted, balanced polyhedral fans
. These objects preserve just enough data about the original varieties
to remain meaningful (e.g.\ dimension, degree, etc.), while discarding
much of their complexity. They are also known in the literature as
\emph{Bieri-Groves sets}~\cite{BG}. We can compute them based on Gr\"obner
theory or valuations, depending on how the classical varieties are
presented.  Gr\"obner techniques are better suited for algebraic
descriptions, while valuations provide the right framework in the
presence of geometric information, e.g. a polynomial
parameterization. The newly develop theory of \emph{geometric
  tropicalization}, introduced by Hacking, Keel and Tevelev~\cite[{$\S
  2$}]{HKT}, fits into the latter.

The crux of geometric tropicalization is to read off the
tropicalization of a smooth closed subvariety of a torus directly from
the combinatorics of its boundary in suitable compactification. To do
so, its boundary is required to have simple normal crossings (snc),
that is, to behave locally like an arrangement of coordinate
hyperplanes. More precisely, let $X\subset \TP^r$ be the subvariety
and pick a normal and $\QQ$-factorial compactification $\overline{X}$
where $X$ is an open subvariety of $\overline{X}$ and its divisorial
boundary $\partial\overline{X}=\overline{X}\smallsetminus X$ is a snc
boundary. The combinatorial information of the tropical variety $\cT
X$ is encoded in an abstract simplicial complex, called the boundary
complex $\Delta(\partial\overline{X})$, which resembles the one
in~\cite{BoundaryComplexes}. After assigning coordinates to the
vertices of this complex by means of divisorial valuations, and
extending linearly on cells, we get a complex in the real span of
the cocharacter lattice of $\TP^n$. Geometric tropicalization says
precisely that the support of the tropical fan is the cone over this
complex and, in particular, the result does not depend on our choice
of $\overline{X}$ (Theorem~\ref{TSGthm:geomTrop}).

The circle of ideas behind geometric tropicalization has deep, yet
not explicit, connections to recent articles on tropical algebraic
geometry. As an example, we can mention the work of M.\ Baker relating
linear systems on curves and linear systems on the dual graphs of
their associated semistable regular
models~\cite{SpecializationLemma}. These dual graphs encode the same
combinatorial information as the boundary complexes used in~\cite[{$\S
  2$}]{HKT}. One possible explanation for this phenomenon is that, up
to now, geometric tropicalization was only able to recover the support
of the tropical fan. Tropical multiplicities were missing from this
description and they are essential for recovering information about
the original algebraic varieties from their tropical counterparts.  By
decoration the boundary complex $\Delta(\partial\overline{X})$ with
weight on its maximal cells, we obtain an \emph{explicit combinatorial
  formula} for computing tropical multiplicities
(Theorem~\ref{SIthm:MainTheoremMultiplicitiesSNC}), complementing the
set-theoretic results of~\cite[{$\S 2$}]{HKT}.


As one main expect, the major difficulty in applying these methods to
compute tropical fans in concrete examples lies in the restrictive
assumptions on the compactification $\overline{X}$. One way of
constructing such an object is provided by a strong resolution of $X$
in the sense of Hironaka~\cite[Theorem
3.27]{KollarResolutionBook}. These resolutions are by no means explicit,
explaining the lack of examples in this theory. The algorithmic difficulties of
performing such a task are numerous and it would be desirable to
attenuated the necessary conditions on $\overline{X}$ to obtain $\cT
X$ from the weighted complex $\Delta(\partial \overline{X})$.  After
studying in detail the surface case, Sturmfels and Tevelev conjectured
that the right condition to impose was not geometric but
combinatorial, requiring the boundary components to intersect in the
expected codimension~\cite{ElimTheory}. They called it the
\emph{combinatorial normal crossing} (cnc) boundary condition. This
property ensures that $\Delta(\partial\overline{X})$ is simplicial
and has the right dimension, namely, one less than the dimension of
$X$. 
In this paper, we prove this conjecture in arbitrary dimension,
addressing the question of tropical multiplicities as well. Here is
the precise statement, which we discuss in Section~\ref{sec:geom-trop}
(Theorem~\ref{thm:CNCforHKT}).
\begin{theoremNN}
  Let $X\subset \TP^n$ be a smooth subvariety and let $\overline{X}$
  be a normal and $\QQ$-factorial compactification with
  \emph{combinatorial normal crossing} boundary $\partial
  \overline{X}$. Then, the weighted set $\cT X$ can be computed from
  $\overline{X}$ using the weighted boundary complex $\Delta(\partial
  \overline{X})$ and the divisorial valuations induced by
  $\partial\overline{X}$.
\end{theoremNN}

\smallskip



Tropical implicitization was pioneered by the work of Sturmfels,
Tevelev and Yu~\cite{NPofImplicitEquation}. Their methods are well
suited for \emph{generic} varieties, and are built on the theory of
geometric tropicalization and the construction of tropical
compactifications in the sense of
Tevelev~\cite{CompactificationsTori}. However, real life is seldom
generic, so it is crucial to attack the \emph{non-generic} versions of
these problems. 
One of the
contributions of this paper is to identify the genericity conditions,
describe certificates for them, and introduce tropical implicitization
methods for non-generic
surfaces. 


In Section~\ref{SIsec:trop-elim-gener} we focus our attention on the
generic setting: surfaces parameterized by Laurent polynomial maps
with fixed support and where we allow the coefficients to vary
generically. Following~\cite{NPofImplicitEquation}, we translate our
tropical implicitization question to the one of compactifying an
arrangement of plane curves in the torus $\TP^2$. These curves are
precisely the vanishing locus of each coordinate of the given
polynomial parameterization.  
The genericity conditions in~\cite{NPofImplicitEquation} are chosen so
that we have a natural choice for our cnc compactification: a smooth
projective toric surface build from the supports of our
parameterization.
Our approach allows to weaken these conditions and still be able to
compute our tropical surface from the input map. As a byproduct, in
Theorem~\ref{thm:TropImplGenericSurfaces} we show that the smoothness
condition on the ambient space is unnecessary to obtain the tropical
surface encoded as a weighted graph.  We illustrate our approach with several
numerical examples in $\afc^3$. These examples are then revisited to
highlight the differences between the techniques applied to generic
and non-generic surfaces.

In Section~\ref{SIsec:trop-elim-non} we discuss tropical
implicitization of non-generic surfaces. We start by clarifying what
we mean by \emph{special} surfaces. Then, we describe a procedure to
obtain the graphs associated to their tropicalization.  Singularities
coming from excessive intersections are the main obstruction to apply
the methods of Section~\ref{SIsec:trop-elim-gener} in this context.
To fix this bad behavior, we first compactify the arrangement of plane
curves inside $\pr^2$. Secondly, since the cnc condition on the
boundary fails to hold, we must resolve the excessive boundary points,
for example, by ordinary blow-ups. This construction yields the
desired cnc compactification. Thus, the corresponding tropical surface
can be obtained using
Theorem~\ref{thm:CNCforHKT}.

We end this paper with some remarks and open questions. As our running
examples illustrate, rational surfaces in $\CC^3$ serve as a nice test
case to explore tropical implicitization techniques. In this setting,
these methods require to analyze the combinatorics of a curve
arrangement in $\TP^2$, and the local behavior at points belonging to
three or more of these curves. Topological methods from singularity
theory can then be applied to predict the resulting tropical
surfaces. Even though the theory of tropical implicitization is at an
early stage and it is still evolving, we expect
Theorems~\ref{SIthm:MainTheoremMultiplicitiesSNC}
and~\ref{thm:CNCforHKT} to become a valuable tool for future
applications.

\section{Geometric tropicalization}
\label{sec:geom-trop}

In this section, we discuss the theory of geometric
tropicalization. We present its original formulation in set-theoretic
terms as in~\cite{HKT}, and we extend it with two results. The first
main result is a \emph{formula for tropical multiplicities}. 
The second one proves a conjecture of Sturmfels and
Tevelev~\cite{ElimTheory} on necessary conditions to compute tropical
varieties from their boundary complexes and their associated
divisorial valuations. More precisely, rather than requiring a simple
normal crossing (snc) boundary, it is enough to require a
\emph{combinatorial normal crossing} (cnc) boundary. The main
advantage of this weaker hypothesis will become evident in
Sections~\ref{SIsec:trop-elim-gener} and~\ref{SIsec:trop-elim-non}.

\begin{notation} Throughout this paper, we fix the following notation.
  Let $\TP^r$ be the $r$-dimensional algebraic torus over a field $k$
  of characteristic zero. Let $\Lambda_r =\HH(k^*, \TP^r)$ be the
  cocharacter lattice and $\Lambda_r^{\vee}=\HH(\TP^r, k^*)$ the
  character lattice. We let $K|k$ be the field of Puiseux
  series with parameter $\varepsilon$ and with valuation
\[
\mathbf{\ord}\colon K\to \RR\cup \{\infty\}, \qquad
\alpha\,\varepsilon^u + (\text{higher order terms}) \mapsto u.
\]
Given a fan $\cF$ in $\RR\otimes \Lambda_r$ we let $X_{\cF}$ be the
associated toric variety, with intrinsic torus $\TP^r$.  As it is
standard in toric geometry, given a cone $\Sigma$ of $\cF$, we let
$V(\Sigma) \subset X_{\cF}$ be the closure of the torus orbit
$\OO_{\Sigma}$, with intrinsic torus $\TP_{\Sigma}$.
\end{notation}

We start by recalling the basics on tropical
geometry~\cite{BJSST}. Our exposition will be coordinate free, but the
reader can safely pick a basis of characters for each $r$-dimensional
algebraic tori and view all tropical varieties in $\RR^r$ rather than
in the $\RR$-span of the cocharacter lattice.

The \emph{tropicalization} of a closed subvariety $X$ of the algebraic
torus $\TP^r$ is a fan in $\RR\otimes \Lambda_r$ with intrinsic
lattice $\Lambda_r$. It is defined as:
\begin{equation}
\cT X = \{w \in \Lambda_r \mid 1\notin \init_w(I_X)\}.\label{eq:1}
\end{equation}
Here, $I_X$ is the defining ideal of $X$ in the Laurent polynomial
ring $k[\Lambda^\vee_r]$, 
and $\init_w(I_X)$ is the ideal of all initial forms $\init_w(f)$ for
$f\in I_X$. The set $\cT X$ is a rational polyhedral fan of dimension
$\dim X$~\cite{BG}. A point $w\in \cT X$ is called \emph{regular} if
there exists a vector subspace $\mathbb{L}_w\subset \RR\otimes
\Lambda_r$ such that $\cT X$ and $\mathbb{L}_w$ agree locally near
$w$. The tropical variety $\cT X$ can be endowed with a locally
constant function called \emph{multiplicity}, defined on regular
points, and that satisfies a \emph{balancing
  condition}~\cite[Definition 3.3]{ElimTheory}. There are many ways of
defining these numbers. For example, $m_w$ can be computed as the sum
of the multiplicities of all minimal associated primes of the initial
ideal $\init_w(I_X)$~\cite[$\S 2$]{TropDiscr}. Similarly, if $\Sigma$
is a cone in $\cT X$ that contains $w$ we can
define $m_w$ as the length of the 0-dimensional scheme
$V({\Sigma})\cap Z$, where $Z$ is the closure of $X$ in the toric
variety associated to the fan $\cT X$~\cite[Lemma 3.2]{ElimTheory}. 
Theorem~\ref{SIthm:MainTheoremMultiplicitiesSNC} gives an alternative
combinatorial approach for obtaining these invariants.

The theory of geometric tropicalization aims to compute tropical
varieties from geometric information on the underlying classical
varieties.  Our main players are the notions of cnc and snc pairs, and
their associated boundary complex. Roughly speaking, starting from a
smooth closed subvariety $X\subset \TP^r$ we find a cnc pair
$(\overline{X},\partial \overline{X})$ and we construct a quotient of
the boundary complex $\Delta(\partial \overline{X})$
from~\cite{BoundaryComplexes}. This simplicial complex collects the
combinatorial structure of the tropical fan $\cT X$.  Here are the
precise definitions:

\begin{definition}\label{def:cncAndSncPair}
  Let $X$ be a smooth subvariety of a torus $\TP^r$, and let
  $\overline{X}$ be a normal and $\QQ$-factorial compactification
  containing $X$ as a dense open subvariety. Let $\partial
  \overline{X} = \overline{X} \smallsetminus X$ be the boundary
  divisor of $\overline{X}$.  We say that this boundary is a
  \emph{combinatorial normal crossings divisor} if for every integer
  $l$, and any choice of $l$ boundary components, their intersection
  has codimension $l$. Similarly, the boundary is a \emph{simple
    normal crossings divisor} if, in addition, this intersection is
  transverse (i.e.\ the intersection behaves locally like a hyperplane
  arrangement).  We say that $(\overline{X},\partial\overline{X})$ is
  a \emph{combinatorial normal crossing pair} or \emph{cnc pair} for
  short, if the boundary is a combinatorial normal crossing
  divisor. \emph{Simple normal crossing pairs} (\emph{snc pairs} for
  short) are defined analogously.
\end{definition}
Note that the normality condition on $\overline{X}$ is imposed so that we can
define the order of vanishing of a rational function along an
irreducible divisor. The $\QQ$-factorial property says that Weil
divisors are $\QQ$-Cartier and it enables us to view Cartier divisors
as a subgroup all Weil divisors, thus, allowing us to speak of
divisors without further distinction. In addition, in this setting,
intersection numbers among boundary components are well
defined~\cite[Chapter 2]{FultonIT}. These numbers will be crucial when
discussing tropical multiplicities.  If $\overline{X}$ is smooth, then
the normality and $\QQ$-factorial conditions are automatically
achieved. In the language of~\cite{CompactificationsTori}, cnc pairs
will yield tropical compactifications of subvarieties of tori.

\begin{definition}\label{def:boundaryComplex}
  Let $(\overline{X}, \partial \overline{X})$ be a cnc pair. The
  \emph{boundary 
  complex} $\Delta(\partial \overline{X})$ is a simplicial complex
  whose  vertices $\{v_1, \ldots, v_m\}$ are in
  one-to-one correspondence with the $m$ components of the boundary
  divisor $\partial\overline{X}=\bigcup_{i=1}^mD_i$. Given a nonempty
  subset $I\subset \{1, \ldots, m\}$, the boundary complex contains
 a cell $\sigma_I$ spanned by $\{v_i: i\in I\}$ if an only if the
 intersection $D_I:=\bigcap_{i\in I}D_i$ is nonempty.
\end{definition}
We should remark that our definition of boundary complex differs from that
of~\cite{BoundaryComplexes} in two ways. First, Payne endows this
complex with a topological structure, and secondly, he picks one
simplex per component of the intersection $D_I$. Instead, we prefer
to identify these simplices with the unique cell $\sigma_I$ and forget about the
topological nature of this complex since our motivation is mainly
combinatorial.Thus, our construction can be naturally viewed as a quotient of
that in~\cite{BoundaryComplexes}.

Our next step is to realize the boundary complex
$\Delta(\partial\overline{X})$ in the cocharacter lattice of the
algebraic torus $\TP^r$, as in~\cite{HKT}. This is done by associating
a point in the lattice $\Lambda_r$ to every vertex $v_i$ of the
complex and extending linearly on higher-dimensional cells, following
the valuative definition of tropical varieties.  Given a component $D$
of $\partial\overline{X}$ we let $\val_D({\underline{\ \ }})$ be the order of
zeros-poles along $D$ of elements in $K[X]$. By construction, $\val_D$
is a valuation on $K[X]$ that restricts to $\mathbf{\ord}$ on
$K$~\cite[Section 2]{ElimTheory}. This valuation specifies an element
$[\val_D]$ of $\RR\otimes \Lambda_r$ by the formula
\begin{equation*}
  \label{eq:2}
  [\val_D](m):= \val_D(m_{|_X}) \qquad \text{ for any } m \in \Lambda_r^{\vee},
\end{equation*}
and extending linearly. If we fix a basis of characters $\{\chi_1,
\ldots, \chi_r\}$ of $\TP^r$, then $[\val_D]$ is identified with a
point in $\ZZ^r$, namely $[D] := (\val_{D}(\chi_1), \ldots,
\val_{D}(\chi_r))$ and $\val_D(\chi_i)$ is the order of vanishing of
$\chi_i$ along $D$.

For any  $\sigma_I \in \Delta(\partial\overline{X})$, let $[\sigma_I]$ be the
semigroup spanned by $\{[\val_{D_i}] : i\in I\}\subset \Lambda_r$. The
realization of $\Delta(\partial \overline{X})$ is the collection
$\{[\sigma_I]: I\}$. We choose the word
``realization'' rather than embedding because this map need not be
injective. As Theorem~\ref{TSGthm:geomTrop} shows, the cone over this
complex in $\RR\otimes \Lambda_r$ does not depend on the choice of the
cnc pair $(\overline{X}, \partial \overline{X})$.

\medskip

The following result of Hacking, Keel and Tevelev says that the
tropical fan $\cT X$ is precisely the cone over the
\emph{realization} of the boundary complex in the cocharacter lattice for a
given a snc pair:

\begin{theorem}[\textbf{Geometric tropicalization {\cite[\S
      2]{HKT}}}] \label{TSGthm:geomTrop} Let  $X$ be a closed smooth 
  subvariety of $\TP^r$. 
  Let $(\overline{X}, \partial \overline{X})$ be a snc pair and
  $\Delta(\partial \overline{X})$ its boundary complex. Then, the
  tropical set $\cT X$ is the cone over the realization of
  $\Delta(\partial \overline{X})$ in the cocharacter lattice of
  $\TP^r$, i.e.
\begin{equation}\label{eq:GeomTrop}
\cT X = \bigcup_{\sigma \in \Delta(\partial\overline{X})} \RR_{\geq
  0}[\sigma]\;\; \subset \RR\otimes \Lambda_r.
\end{equation}
\end{theorem}
As it is pointed out in~\cite[Remark 2.7]{ElimTheory}, the proof
in~\cite{HKT} shows that the right-hand side of~\eqref{eq:GeomTrop}
contains $\cT X$ if $\overline{X}$ is normal, without any smoothness
or snc pair condition. But this containment can be strict if
$(\overline{X}, \partial \overline{X})$ is not a cnc pair, since it
could include cones of dimension greater than $\dim X$, violating the
Bieri-Groves' Theorem ~\cite{BG}. We will come back to this point in
Theorem~\ref{thm:CNCforHKT}.

\medskip

We now turn into the question of tropical multiplicities.  Consider a
monomial map $\alpha\colon \TP^r \to \TP^n$ associated to an $n\times
r$ integer matrix $A$. We think of this map as a linear map between
the associated cocharacter lattices $A\colon \Lambda_r \to
\Lambda_n$. By~\cite[Theorem 3.12]{ElimTheory} we know that
tropicalization is functorial with respect to monomial maps and
subvarieties of tori, which in particular says that
$\cT(\overline{\alpha(X)})=A(\cT X)\subset \RR\otimes \Lambda_n$.

Assume that $\alpha_{|_{X}}$ has generic fibers of
finite size $\delta$. Under this condition, \cite[Theorem
3.12]{ElimTheory} gives a way of computing multiplicities on
$\cT(\overline{\alpha(X)}) $ from the multiplicities on $\cT X$, the
degree $\delta$ and the fibers of $A$, known as the \emph{push-forward
  formula for multiplicities} of Sturmfels-Tevelev. Namely,
\begin{equation}
m_w = \frac{1}{\delta} \, \sum_v m_v \; \operatorname{ index
}\,(\mathbb{L}_w \cap \Lambda_n ,A(\mathbb{L}_v \cap \Lambda_r)),\label{eq:ST}
\end{equation}
where we sum over all points $v \in \cT X$ with $A v = w$, which are
assumed to be finite and regular. Here, $\mathbb{L}_v$ and
$\mathbb{L}_w$ are the linear spans of neighborhoods of regular points
$v \in \cT X$ and $w \in A (\cT X)$, respectively.

We now state the first main result in this section: a combinatorial
formula for computing tropical multiplicities, complementing
Theorem~\ref{TSGthm:geomTrop}. In the complete intersection case, our
theorem is equivalent to~\cite[Theorem 4.6]{ElimTheory}.  The index
factor accounts for the change in the lattice structure from the
sublattice $\ZZ[\sigma]$ to its saturation $\RR 
[\sigma]\cap \Lambda_r$ in $\Lambda_r$.

\begin{theorem}\label{SIthm:MainTheoremMultiplicitiesSNC}
  Let $X\subset \TP^r$ be a smooth $s$-dimensional closed subvariety
  and let $(\overline{X}, \partial \overline{X})$ be a snc pair.
  Then, the multiplicity of a regular point $w$ in the tropical
  variety $\cT X$ equals
  \begin{equation}
    m_w = \sum_{\sigma}
    (
    D_{k_1} \cdot \ldots \cdot D_{k_s}) \;\operatorname{
      index}\big(\RR 
    [\sigma] \cap \Lambda_r ,
    \ZZ[\sigma]\big)
    ,\label{SIeq:9}
\end{equation}
where $D_{k_1} \cdot \ldots \cdot D_{k_s}$ denotes the intersection
number of these $s$ divisors and we sum over all $(s-1)$-dimensional
cells $\sigma=\{v_{k_1}, \ldots, v_{k_{s}}\}$ in $\Delta(\partial\overline{X})$ whose associated cone
$\RR_{\geq 0} [\sigma]$ has dimension $s$ and contains $w$.
\end{theorem}

\begin{proof}
  Since our question is local, it suffices to show that the result
  holds for a choice of a snc pair $(\overline{X}, \partial
  \overline{X})$ whose underlying boundary complex gives a rational
  polyhedral fan in $\RR\otimes \Lambda_r$, rather than just a
  collection of cones that supports $\cT X$. For example, we could
  pick $\overline{X}$ to be the toric variety associated to a smooth
  structure on $\cT X$ (a refinement of Tevelev's tropical
  compactification~\cite{CompactificationsTori}).  In this setting,
  each regular point of $\cT X$ comes from a single top-dimensional
  cell $\sigma$ of $\Delta(\partial\overline{X})$. The general
  formula~\eqref{SIeq:9} is then a direct consequence of the
  additivity of tropical multiplicities~\cite[Construction
  2.13]{AllermanRau}.

  Our strategy goes a follows. We start by fixing a smooth fan
  structure on $\cT X$ that is compatible with $\Delta(\partial
  \overline{X})$. Then, for each maximal cone $\Sigma$ in this fan, we
  consider the codimension $s$ torus $\TP_{\Sigma}$ and we relate the
  tropical variety $\cT X$ to $\cT (X\cap \TP_{\Sigma})$ via the
  inclusion of tori $\TP_{\Sigma} \hookrightarrow \TP^r$. Since the
  multiplicity of every regular point in $\Sigma$ is be the
  intersection number of the $s$ boundary divisors associated to
  $\Sigma$, formula~\eqref{SIeq:9} follows from the push-forward
  formula~\eqref{eq:ST}.

  Let us further explain the previous outline.  By standard arguments
  in geometric combinatorics, we can extend the tropical fan in
  $\RR\otimes \Lambda_r$ to a complete fan $\cF$.  We pick a regular
  point $w$ of $\cT X$ and we let $\Sigma$ be the unique maximal cone
  of $\cT X$ containing $w$. By our assumption on the pair
  $(\overline{X}, \partial \overline{X})$, this cone can be written as
  $\RR_{\geq 0}[\sigma]$ for a unique maximal cone $\sigma \in
  \Delta(\partial \overline{X})$. If $Z$ is the closure of $X$ in the
  toric variety $X_{\Sigma}$ we know that
  $Z\cap V(\Sigma)$ is a zero-dimensional scheme of
  length $m_w$.  This number equals the intersection product of the
  cycles $Z$ and $V(\Sigma)$ in $X_\cF$~\cite[Lemma
  3.2]{ElimTheory}.

  By~\cite[Lemma 2.2]{CompactificationsTori} we know that $Z$ does not
  intersect codimension $s+1$ toric strata of $X_\cF$. In particular,
  $Z\cap \TP_{\Sigma}=Z\cap V(\Sigma)$ is nonempty: it is a complete
  intersection defined by the $s$ divisors $\{D_{1}, \ldots, D_{s}\}$
  in $\TP_{\Sigma}$ associated to $\sigma=\{v_1, \ldots, v_s\}$.  The
  length of this scheme equals the intersection number of these $s$
  divisors and it agrees with the multiplicity of $w$ as a point of
  $\cT (X\cap \TP_{\Sigma})\subset \RR[\sigma]$.  Using the
  push-forward formula~\eqref{eq:ST} for the monomial map
  $\TP_{\Sigma}\hookrightarrow \TP^r$, the multiplicity of $w$ in $\cT
  X$ equals the intersection number of the $s$ divisors $D_1, \ldots,
  D_s$ times the index of the lattice $\ZZ[\sigma]$ in its saturation
  $\RR[\sigma]\cap \Lambda_r$.  This concludes our proof.
\end{proof}

The previous theorem allows us to endow the boundary complex
$\Delta(\partial \overline{X})$ with weights on its maximal
cells. More precisely, a maximal cell $\sigma_I=\{v_{i_1}, \ldots,
v_{i_s}\}$ gets weight $m_{\sigma_I}:=D_{i_1}\cdot\ldots \cdot
D_{i_s}$.  The realization of this complex inherits these weights in
the expected way, namely
\begin{equation*}m_{[\sigma_I]}:=(D_{i_1}\cdot\ldots \cdot D_{i_s}) \,\operatorname{index}(\RR[\sigma_I]\cap \Lambda_r,
\ZZ[\sigma_I]). 
\end{equation*}
Theorem~\ref{SIthm:MainTheoremMultiplicitiesSNC} says that the
multiplicity of a regular point $w$ in $\cT X$ is obtained by summing
up the weights of the cones over all maximal cells $[\sigma_I]$ that
contain $w$.

\begin{example} \label{ex:SpecialSurface} Consider the plane $X$ of
  $\TP_{\CC}^3$ defined by the equation $x+y+z+1=0$. We compactify $X$
  in $\pr^3$. Then, $(\overline{X}, \partial{\overline{X}})$ is a snc
  pair whose boundary complex is the 1-skeleton of the $3$-dimensional
  simplex (on the left of Figure~\ref{fig:specialPlane}), with
  constant weight one. Its vertices are $(1,0,0)$, $(0,1,0)$,
  $(0,0,1)$ and $(-1,-1,-1)$, so we recover the expected 
  generic tropical plane in $\RR^3$.
\end{example}
\begin{example}\label{ex:SpecialSurfaceBlowup} We now pick the special plane in $\TP_{\CC}^3$ with
  equation $x+y+z=0$. Its compactification in $\pr^3$ has four
  components (three lines), but three of them intersect at the point
  $(0:0:0:1)$. The boundary complex is shown on the left of
  Figure~\ref{fig:specialPlane}. If we blow up this point, we obtain a
  new compactification of $X$ with five components. The boundary
  complex is a graph with five vertices and constant weight one, shown
  on the right of Figure~\ref{fig:specialPlane}. It is obtained from
  the boundary complex on the left by replacing the unique 2-cell by a
  subdividing tripod tree, whose inner vertex $E$ corresponds to the
  exceptional divisor with $[E]=(1,1,1)$. 
 \begin{figure}
   \centering
   \includegraphics[scale=0.6]{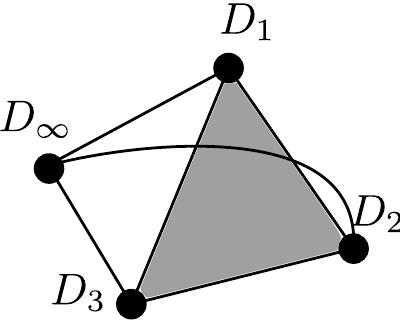}
\qquad \quad \includegraphics[scale=0.6]{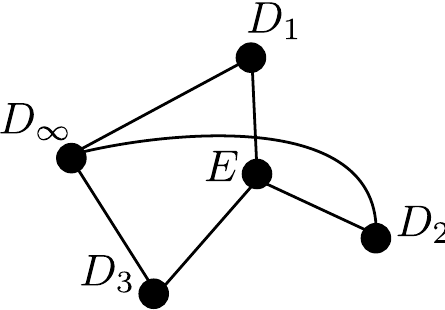}
   \caption{Boundary complexes associated to the plane $x+y+z=0$ in $\TP^3$.}
   \label{fig:specialPlane}
 \end{figure}
    \end{example}

    Note that all the results that we have stated so far are for snc
    pairs. But it would be desirable to weaken this strong condition
    on $\overline{X}$.  By the Bieri-Groves theorem~\cite[Theorem
    4.5]{BG}, the dimension of $\cT X$ equals $\dim X$. 
    By construction, a cnc pair yields a collection of cones in
    $\RR\otimes \Lambda_r$ of the expected dimension. This condition
    was violated in Example~\ref{ex:SpecialSurfaceBlowup} for the
    naive compactification in $\pr^3$. In~\cite{ElimTheory}, Sturmfels
    and Tevelev conjectured that this condition is also sufficient for
    computing supports of tropical varieties, confirming this result
    in the surface case~\cite[Proposition 5.4]{ElimTheory}. We prove
    this conjecture in any dimension, incorporating tropical
    multiplicities into the statement.

\begin{theorem}\label{thm:CNCforHKT} Let $(\overline{X}, \partial
  \overline{X})$ be a cnc pair. Then, the cone over the weighted
  realized boundary complex $\Delta(\overline{X})$ supports the
  weighted fan
  $\cT X$. 
  \end{theorem}
   \begin{proof} 
     If $(\overline{X}, \partial \overline{X})$ is a snc pair, then
     the result follows by Theorems~\ref{TSGthm:geomTrop}
     and~\ref{SIthm:MainTheoremMultiplicitiesSNC}. If this condition
     is not satisfied, we must modify this cnc pair to obtain a new
     one $(\overline{X}', \partial{\overline{X}'})$ that is a snc
     pair. This modification is done by resolving the variety
     $\overline{X}$ until the snc condition is achieved, by means of
     Hironaka's strong resolution of singularities~\cite[Theorem
     3.27]{KollarResolutionBook}. Our goal is to show that the cones
     over the weighted realized boundary complexes $\Delta(\partial
     \overline{X})$ and $\Delta(\partial\overline{X}')$ agree. We
     divide the proof into two parts: the set-theoretic identity and
     the multiplicity statement. This is the content of
     Lemmas~\ref{lm:SetIdentity} and~\ref{lm:MultIdentity}.
\end{proof}

\begin{lemma}\label{lm:SetIdentity}
  The support of the cone over the realized boundary complex of a cnc
  pair is invariant under resolutions.
\end{lemma}
\begin{proof}
  Let $s$ be the dimension of $X$. The boundary $\partial\overline{X}$
  may contain singularities. Let $Z_1$ be the set of singular points
  contained in at most $s-1$ boundary divisors and $Z_2$ the set of
  boundary points where $s$ boundary components do not meet
  transversally. Since $(\overline{X}, \partial{X})$ is a cnc pair, we
  know that $Z_2$ is a finite set of points. We define $Z=Z_1\cup Z_2$
  and we show that under a resolution of $\overline{X}$ along $Z$, the
  cones over the realizations of the new boundary complex coincides
  with the cones over $\Delta(\partial \overline{X})$. Roughly
  speaking, rays associated to $[\val_E]$ for an exceptional divisor
  $E$ will not change the support of the fan associated to the
  original. It suffices to deal with $Z=Z_1$ or $Z=Z_2$
  separated. Furthermore, since the question is local, we may assume
  $Z$ is irreducible.

  Suppose $Z=Z_1$, and denote by $\pi\colon\overline{X}'\to
  \overline{X}$ the resolution of $\overline{X}$ along $Z$. By working
  over an open cover, we may assume that $Z$ corresponds to the
  intersection of $l$ divisors, namely $D_1, \ldots, D_k$. Set $I=\{1,
  \ldots, l\}$. Note that $l\leq s-1$. For each $i\in I$, let $D'_i$
  be the strict transform of $D_i$ and $E_1, \ldots, E_u$ be the
  exceptional divisors. By construction, $\pi^*(D_i)=D_i'+ m_{i1}\,
  E_1 + \ldots +m_{iu} \,E_u$ for all $i\in I$, with $m_{ij}>0$ for
  all $i,j$, and the boundary
  complex $\Delta(\partial\overline{X}')$ is obtained from
  $\Delta(\partial\overline{X})$ by relabeling the vertices $v_i$ by
  $v_i'$ ($i\in I$), adding the vertices $e_1, \ldots, e_u$ associated
  to $E_1, \ldots, E_u$ and replacing the cell $\sigma_I$ by a complex
  subdividing $\sigma_I$. This complex contains cells of dimension at
  most $l-1$ with vertices in $\{v_i': i\in I\} \cup \{ e_1, \ldots,
  e_u\}$.  Notice that the divisorial valuations satisfy
  $\val_{D_i}=\val_{D_i'}$ and $\val_{E_j}=\sum_{i\in
    I}m_{ij}\,\val_{D_i}$, $j=1, \ldots, u$. Thus, the support of the
  cone over 
  the subdivided $[\sigma_I]$ is contained in $\RR_{\geq
    0}[\sigma_I]$. This cone has dimension at most $s-1$ so it does not
  contribute to $\cT X$.


  Next, assume $Z=Z_2$ is a point in $D_I=\bigcap_{i\in I} D_i$ for
  $|I|=s$. Since the question is local, we may assume that the
  boundary $\partial\overline{X}$ consists of these $s$ divisors whose
  intersection is supported at a single point $p$ (possible with
  multiplicity).  In this situation, the boundary complex
  $\Delta(\partial\overline{X})$ is an $(s-1)$-dimensional simplex,
  with vertices $\{v_i: i \in I\}$.  Keeping the notation from the
  case $Z=Z_1$, the resolution $\pi\colon\overline{X}'\to
  \overline{X}$ at the point $p$ gives
\[
\pi^*(D_i)=D_i'+ \sum_{j=1}^um_{ij}\,E_j \qquad i\in I,
\]
where all $m_{ij}$ are positive integers and $E_1, \ldots, E_u$ are
the components of the exceptional locus.  As before, we have
$\val_{D_i'}=\val_{D_i}$, $\val_{E_j}=\sum_{i\in I} m_{ij}\,
\val_{D_j}$, for all $j=1, \ldots, u$. In particular, $[\val{E_j}]\in
\RR[\sigma_I]$.  If the valuations $\{[\val_{D_i}]: i\in I\}$ are
linearly dependent, the cones over the realizations of
$\Delta(\partial\overline{X})$ and $\Delta(\partial\overline{X}')$
have dimension at most $s-1$, and there is nothing to prove. Thus, we
may assume the valuations $\{[\val_{D_i}]: i\in I\}$ are linearly
independent.


The boundary complex $\Delta(\partial \overline{X}')$ has $s+u$ vertices
$\{v_i: i\in I\}\cup \{e_1, \ldots, e_u\}$. To simplify notation, we
replace this complex by the $s$-dimensional weighted flag complex $\Gamma$ on
these $s+N$ vertices, with weights on maximal cells given by the intersection
number of the associated divisors.  If these divisors do not meet, the
given weight is zero, and we know that this cell does not belong to
$\Delta(\partial \overline{X}')$. 
By Theorem~\ref{TSGthm:geomTrop}, we know that $[D_i']\in \cT X$ for
all $i\in I$.  The support of $\cT X$ contains the cone spanned by
these $s$ rays if and only if regular point in $\Gamma$ has
positive weight. Lemma~\ref{lm:MultIdentity} shows that this weight
equals the intersection number of the divisors $\{D_i: i\in I\}$, which
is positive by hypothesis. This concludes our proof.
\end{proof}

\begin{lemma}\label{lm:MultIdentity}
The weights of the realized boundary complex of a cnc pair are invariant under
resolutions.
\end{lemma}

\begin{proof}
  We keep the notation of Lemma~\ref{lm:SetIdentity}. The cones coming
  from a resolution of $Z=Z_1$ are not maximal, so they do not
  contribute any weights. Thus, we only need to analyze the case
  $Z=Z_2$.

\begin{figure}[htb]
  \centering
  \includegraphics[scale=0.35]{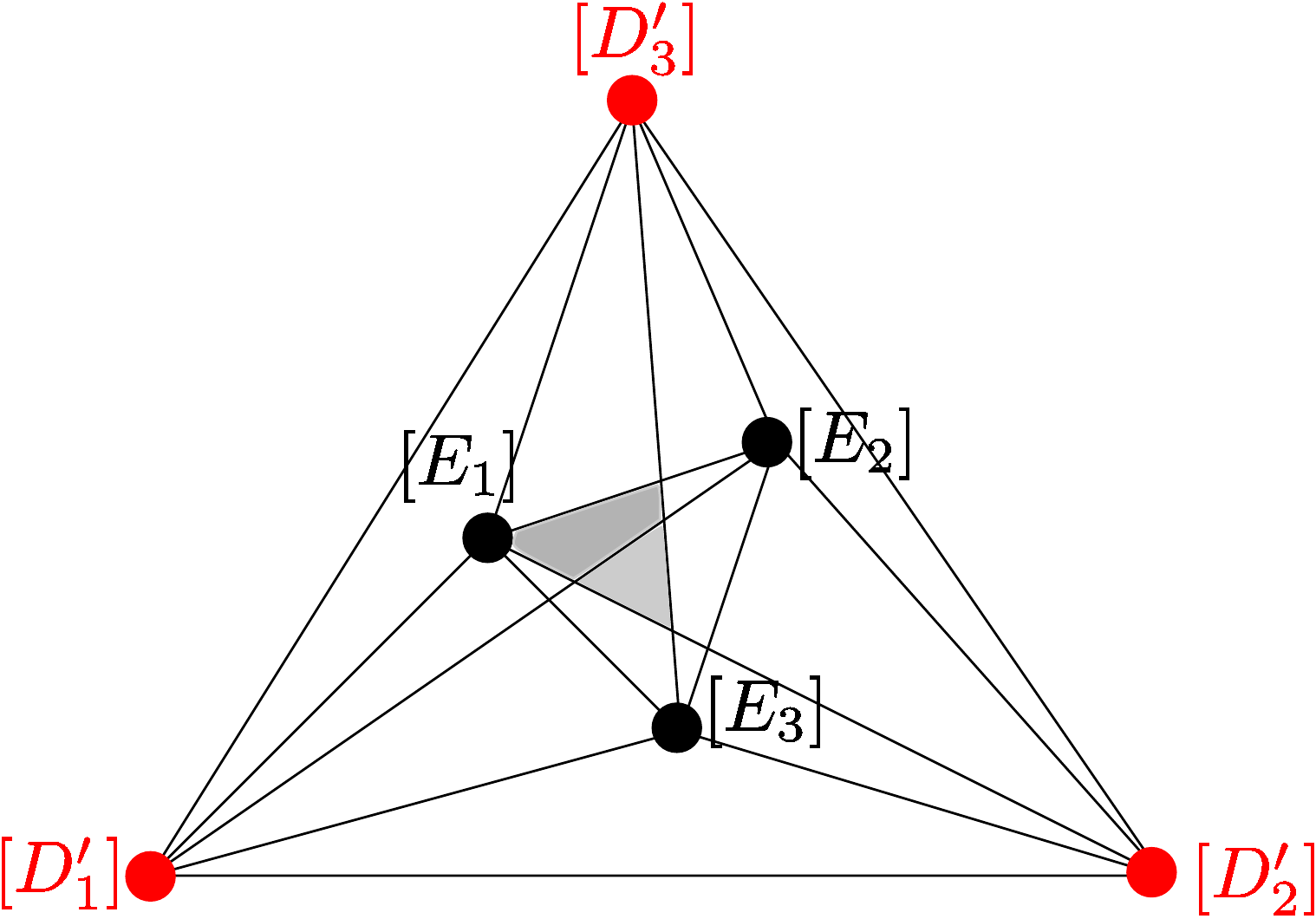}
  \caption{Divisorial valuations and flag complex $\Gamma$ arising
    from the resolution of the boundary $\bigcup_{i\in I}D_i$ at
    $p$ when $|I|=3$.  }
\label{fig:BlowupAndGeomTrop}
\end{figure}

After picking a basis of the saturation in $\Lambda_r$ of the rank $s$
sublattice generated by $[\sigma_I]$, we may assume that
$[D_i]=[D_i']=d_ie_i$ for all $i\in I$. Since the right-hand side of
formula~\eqref{SIeq:9} is multiplicative with respect to each $d_i$,
we may assume that $d_i=1$ for all $i \in I$. In this new coordinate
system, we have $[E_j]=(m_{1j}, \ldots, m_{sj})$ in $\ZZ^s$ for all
$j=1, \ldots, u$, as in Figure~\ref{fig:BlowupAndGeomTrop}.

Following Lemma~\ref{lm:SetIdentity}, we wish to show that all regular
points in the cone over the flag complex $\Gamma$ have weight $D_1\cdot
\ldots \cdot D_s$. Consider a coarse fan structure $\cF$ on the
$s$-dimensional cone over the realization of $\Gamma$. For example, in
Figure~\ref{fig:BlowupAndGeomTrop}, this corresponds to having 12
vertices on the spherical complex induced by $\Gamma$: the six big
dots, together with the six crossings of edges induced by the realization
of $\Gamma$. 
Notice that the hyperplanes supporting facets in $\cF$ are
spanned by subsets of $\{[D_i']: i\in I\} \cup \{[E_1], \ldots,
[E_u]\}$.

%
%
%

By construction, there are two types of cones to consider: the ones
with a supporting hyperplane facet spanned by $s-1$ vertices in
$\{D_i': i\in I\}$, and the ones that do not have this property. We
prove our claim by a \emph{wall-crossing type formula} in two
steps. First, we show that the first type of cones have the expected
weight. Second, we show that if a cone has the expected weight, the
same is true for all its neighbors. Since the fan $\cF$ is connected
in codimension one by construction and the facets of 
$\sigma_I$ are generated by $s-1$ vertices in $\{v_i: i\in I\}$, this
will prove our statement.

For simplicity, assume $I=\{1, \ldots, s\}$ and pick a
cone spanned by $\{[D_{1}'], \ldots, [D_{{s-1}}'], [E_j]\}$ for some
$j$. By formula~\eqref{SIeq:9}, we know that the weight of this cone is
the sum of the weight of all cells in $\Gamma$ whose cones contain the
former. In particular, 
\begin{equation}
D_1'\cdot \ldots \cdot D_s' +\sum_{k=1}^u m_{sk}\,D_1'\cdot \ldots
\cdot D_{s-1}' \cdot E_k = D_1'\cdot \ldots \cdot D_{s-1}' \cdot
(D_s'+ \sum_{k=1}^u m_{sk}\,
E_k)=D_1'\cdot \ldots \cdot D_{s-1}' \cdot
 \pi^*(D_s),\label{eq:3}
\end{equation}
since $m_{sk}$ is the index of the lattice spanned by $\{[D_1'],
\ldots, [D_{s-1}'], [E_j]\}$ in its saturation.  Since the resolution $\pi$ is a
proper morphism, using the projection formula we have $\pi^*(D_1)\cdot
\ldots \cdot \pi^*(D_s)= D_1\cdot \ldots \cdot D_s$ and
$\pi^*(D_i)\cdot W=0$ for all divisor $W$ contained in the exceptional
locus. Thus, expression~\eqref{eq:3} equals $D_1\cdot \ldots \cdot
D_s$ and so the first type of cones has the expected weight. \smallskip

Now, pick two neighboring cones $\CF_1$ and $\CF_2$ with multiplicities
$m_{\CF_1}$ and $m_{\CF_2}$ that intersect at a common facet $F$. For
example, the shaded cells in Figure~\ref{fig:BlowupAndGeomTrop}. Our
goal is to show that $m_{\CF_1}=m_{\CF_2}$. For each $I_0\subset \{1,
\ldots, s\}$ and $J_0\subset \{1, \ldots, u\}$ we call
$D_{I_0}=\{D'_i: i\in I_0\}$, $E_{J_0}=\{E_j: j\in J_0\}$ and
$[D_{I_0}], [E_{J_0}]$ the associated semigroups in $\ZZ^s$.  We consider
all pairs $(I_0, J_0)$ such that the facet $F$ lies in the
$(s-1)$-dimensional cone spanned by $[D_{I_0}] \cup [E_{J_0}]$, where
$|I_0|+|J_0|=s-1$.  By construction, every cone over a cell of
$\Gamma$ containing $\CF_1$ either also contains $\CF_2$ or it
intersects $\CF_2$ only at the face $F$. Thus, we divide all maximal
cones of $\RR_{\geq 0} \Gamma$ into four types: the ones containing
$\CF_1 \cup \CF_2$, the ones containing $\CF_1$ and not $\CF_2$, the
ones containing $\CF_2$ but not $\CF_1$, and the ones containing
neither $\CF_1$ nor $\CF_2$ (see
Figure~\ref{fig:BlowupAndGeomTrop}). Cones of types two and three are
spanned the $s-1$ rays in $[D_{I_0}] \cup [E_{J_0}]$, together with an
extra ray. Formula~\eqref{SIeq:9} yields

\begin{equation}
m_{\CF_1}-m_{\CF_2}=\sum_{I_0,J_0}\big(\sum_{\substack{\CF \text{ of
      type 2}\\ \RR_{\geq 0} \langle
[D_{I_0}], [E_{J_0}]\rangle \prec \CF}} m_{\CF} \quad- \sum_{\substack{\CF \text{ of
      type 3}\\ \RR_{\geq 0}\langle [D_{I_0}], [E_{J_0}]\rangle \prec \CF}} m_{\CF}\big),\label{eq:8}
\end{equation}
where $\CF$ is the cone over a maximal cell in $\Gamma$ and $\prec$
denotes the order in the face lattice of $\cF$. Notice that the cone
spanned by $[D_{I_0}]\cup [E_{J_0}]$ is a facet of $\CF$.  We prove
that $m_{\CF_1}-m_{\CF_2}$ is zero by showing that for each pair
$(I_0,{J_0})$ the expression between parenthesis in~\eqref{eq:8}
equals zero. Note that only cones of types two and three are
involved.

First, we compute the weights of the cones $\CF$. By definition, we
have
\[
m_{\CF}=
\begin{cases}
D'_{I_0}\cdot E_{J_0}\cdot D'_k\,   |\det([D_{I_0}]|[E_{J_0}]|[D'_k])|
& \text{ if } \CF=\RR_{\geq 0} \langle  [D_{I_0}], [E_{J_0}],
[D'_k]\rangle, k\in I,\\
D'_{I_0}\cdot E_{J_0}\cdot E_l \, |\det([D_{I_0}]|[E_{J_0}]|[E_l])|   & \text{ if }\CF= \RR_{\geq 0}
\langle  [D_{I_0}], [E_{J_0}],[E_l]\rangle, l=1, \ldots u.
\end{cases}
\]
Here $D'_{I_0}\cdot E_{J_0}\cdot D'_k$ (resp.\ $D'_{I_0}\cdot
E_{J_0}\cdot E_l$) denotes the intersection number of the divisors in
$D_{I_0}\cup E_{J_0}\cup \{D'_k\}$ (resp.\ $D_{I_0}\cup E_{J_0}\cup \{E_l\}$).

Fix a pair $(I_0,{J_0})$ such that the facet $F=\CF_1\cap \CF_2$ lies
in the span of $[D_{I_0}]\cup [E_{J_0}]$. To simplify notation, assume
$I_0$ consists of the last $|I_0|$ indices of $\{1, \ldots, s\}$. We
fix the standard orientation of $\RR^s$ and we label the set ${J_0}$
so that the ordered set $D_{I_0}\cup E_{J_0}$ satisfies that $\CF_1$
lies in the positive half-space $F^+$ determined by the linear span of
$F$, whereas $\CF_2$ lies in the negative half-space $F^-$. This
ensures that the determinant in the expression of the multiplicity
$m_{\CF}$ of a cone $\CF$ of type two is positive, whereas for a cone of
type three, this determinant is negative.

For any finite pair of ordered sets $A,B$, we let $\Sn(A,B)$ be set of
injective functions from $A$ to $B$. Each element 
of
$\Sn({A,B})$ has a sign induced by the corresponding element of the
symmetric group on $|A|$ elements.  Fix a cone $\CF$ of type two
spanned by $[D_{I_0}]\cup [E_{J_0}] \cup\{[D_k']\}$. Then, by
expanding the determinant along the column associated to
$[D_k']$, the multiplicity of $\CF$ equals
\begin{equation*}
  \begin{aligned}
    m_{\CF}&= D_{I_0}\cdot E_{J_0} \cdot D'_k \,
    (-1)^{s-|I_0|+1}\sum_{\alpha \in \Sn(J_0,I\smallsetminus (I_0\cup
      \{k\}))} (-1)^{k+s-|I_0|} (-1)^{\sign(\alpha)}(\prod_{j\in
      {J_0}}
    m_{\alpha(j)j})\\
    &= (-1)^{1+k} \sum_{\alpha \in \Sn(J_0,I\smallsetminus (I_0\cup
      \{k\}))} (-1)^{\sign(\alpha)} D_{I_0} \cdot \prod_{j\in {J_0}}
    (m_{\alpha(j)j} E_j)\cdot D_{k}' .
  \end{aligned}
\label{eq:9}
  \end{equation*}
  Likewise, by expanding determinants along the column $[E_l]$,  a cone $\CF$ of type two spanned by $[D_{I_0}]\cup
  [E_{J_0}] \cup \{[E_l]\}$ has multiplicity:
\begin{equation*}
  m_{\CF}
  =  \sum_{k\notin I_0} (-1)^{k+1}\sum_{\alpha \in \Sn(J_0, I\smallsetminus
      I_0)}  (-1)^{\sign(\alpha)} D_{I_0} \cdot \prod_{j\in {J_0}}
  (m_{\alpha(j)j}  E_j) \cdot (m_{kl}E_l). 
\label{eq:10}
  \end{equation*}
The formulas for the multiplicities of cones of type three will
deferred from the previous ones in a sign, due to the orientation convention.

Notice that the previous formulas give the value zero when applied to
cones that lie in the span of $[D_{I_0}]\cup [E_{J_0}]$.  Therefore,
if we fix $(I_0, J_0)$ and we add the contributions to~\eqref{eq:8} of
the cones spanned by $[D_{I_0}]\cup [E_{J_0}]\cup \{[D_k']\}$ and the
cones spanned by $[D_{I_0}]\cup [E_{J_0}]\cup \{[E_l]\}$ for all
$k\notin I_0$ and $l=1, \ldots u$, we obtain
\begin{equation*}
  \sum_{k\in I\smallsetminus I_0}\;\sum_{\alpha \in
    \Sn(J_0, I\smallsetminus
        (I_0\cup \{k\}))}  (-1)^{1+k+\sign(\alpha)}\,  D_{I_0}\cdot \underbrace{(D_k'+\sum_{l=1}^u m_{kl}{ E_l})}_{=\pi^*(D_k)}\cdot
    \prod_{j\in {J_0}} (m_{\alpha(j)j}E_j).\label{eq:4}
  \end{equation*}
  By the projection formula, the previous expression equals 0. This concludes our proof. 
  \end{proof}

\section{Tropical elimination and tropical implicitization}
\label{sec:tropical-elimination}

In this section we discuss tropical elimination and implicitization
theory from the perspective of geometric tropicalization. Our
exposition is based on \cite[Section 5]{ElimTheory} and
\cite{NPofImplicitEquation}. The overall spirit of \emph{tropical
  elimination} lies in computing the tropicalization of the projection
of a variety in $\TP^{r}$ to a coordinate subspace
$\TP^n$. 
%
%
Tropical implicitization is a special instance of tropical
elimination, where our (closed) input variety $X'$ is the graph of a
parameterization given by $n$ Laurent polynomials $\mathbf{f}=(f_1,
\ldots, f_n) \colon X\subset \TP^d\to \TP^n$, i.e.
\[X' := 
{\{(x, \mathbf{f}(x)) : x\in X\}} \subset \TP^{d+n},\] 
and the monomial map $\alpha$ is the projection to the last $n$
coordinates of $\TP^{d+n}$:
\begin{equation}
  \label{SIeq:14bis}
\begin{minipage}[c]{.05\linewidth}
  \xymatrix{\hspace{-5ex}
\TP^d\supseteq X\ar[r]^-{\mathbf{f}}\ar@<-0.5ex>@{^{(}->}[d]_{(id,\mathbf{f})} & 
\overline{\mathbf{f}(X)} \ar@{^{(}->}[d]^{id}  &\hspace{-5ex}\subset
  \TP^n\\
\hspace{-8ex} \TP^{d+n} \supseteq  X'\ar[r]^-{\alpha} &
\TP^n
}
\end{minipage}
  \begin{minipage}[r]{0.12\linewidth}
\hspace{2ex}\xymatrix{ \ar
[rr]^{trop}& &}
  \end{minipage}
  \begin{minipage}[c]{.3\linewidth}
  \xymatrix{
{\RR\otimes \Lambda_d\supseteq \cT X}\ar[r] 
\ar@<5ex>
@{^{(}->}[d]
& 
*+[r]{\cT\overline{\mathbf{f}(X)} }
 \ar@{^{(}->}[d]^{id}  \\
 {\RR\otimes \Lambda_{d+n} \supseteq \cT X'}\ar
 [r]^-{A} &
*+[r]{A(\cT X') \subset \RR\otimes \Lambda_n.}
}
  \end{minipage}
\end{equation}
We aim to compute the tropical
variety $\cT \,\overline{\mathbf{f}(X)}$ from the geometry of $X$ and
the polynomial map $\mathbf{f}$.  For simplicity, we assume
$\mathbf{f}$ is a generically finite map on $X$ of degree $\delta$. In
what follows, we explain how to compute $\cT\,
\overline{\mathbf{f}(X)}$ from $\cT X$ and the projection $\alpha$.

From now on, we fix $Y= \overline{\mathbf{f}(X)} \subset \TP^n$. The
variety $X' \subset \TP^{d+n}$ is a \emph{complete intersection}. If
we fix a basis of characters of $\TP^{d+n}$, this variety is defined
by the ideal $(y_1 - f_1(x), \ldots, y_n - f_n(x))$ 
in $\CC[x_1^{\pm 1}, \ldots, x_d^{\pm 1}, y_1^{\pm 1}, \ldots,
y_n^{\pm 1}]$.  It is isomorphic to $X\subset \TP^n$ via a monomial
map and it projects to $Y$ through the dominant monomial map
$\alpha$. Thus, tropical implicitization reduces to the task of
computing $\cT X'$, which we do by means of geometric tropicalization.

Since $X\subset \TP^d$ and
$X'\subset \TP^{n+d}$ are isomorphic
, we can choose to find a cnc pair for $X$ or $X'$ and build the
corresponding boundary complexes $\Delta(\partial\overline{X})$ or
$\Delta(\partial\overline{X'})$. The realization of the boundary
complex in either $\Lambda_d$ or $\Lambda_{n+d}$ will reflect our
choice. However, since $X$ is not a closed subvariety of $\TP^d$ we
would need to justify the correctness of this step. We do so in the
proof of Theorem~\ref{thm:ImplicitizationIsMonomial}, whose set
theoretic statement appeared already in~\cite[Corollary
2.9]{ElimTheory}.

As in the previous section, we build a cnc pair
$(\overline{X}, \partial \overline{X})$ and its associated weighted
boundary complex, of dimension $d-1$. The novelty with respect to the
previous section will be our choice for a realization of this weighted
complex in the cocharacter lattice $\Lambda_n$. A vertex $v_i$ of
$\Delta(\partial \overline{X})$ gets assign the cocharacter
$[\tilde{D_i}]:=\mathbf{f}^{\#}([D_i])=\val_{D_i}({\underline{\ \ }}
\circ \mathbf{f})$, mapping a character $\chi$ to the lattice point
$\val_{D_i}(\chi \circ \mathbf{f})$. If we fix a basis $\{\chi_1,
\ldots, \chi_n\}$ of characters in $\TP^n$, the resulting cocharacter
is represented by the lattice point $(\val_{D_i}(f_1), \ldots,
\val_{D_i}(f_n))$. The realization of a maximal cell $\sigma_I\in
\Delta(\partial \overline{X})$ in $\Lambda_n$ is the semigroup
$[\tilde{\sigma_I}]$ spanned by $\{[\tilde{D_i}]: i\in I\}$. Note that
the rank of $[\tilde{\sigma_I}]$ may drop. If this is not the case, we
endow the semigroup $[\tilde{\sigma_I}]$ indexed by $I=\{i_1, \ldots,
i_d\}$ with the integer weight
\begin{equation} \label{SIeq:21}
m_{[\tilde{\sigma_I}]}=\frac{1}{\delta}\,(D_{i_1}\cdot \ldots\cdot D_{i_d}) 
\,\operatorname{index}(\RR[\tilde{\sigma_I}]\cap \Lambda_n,
\ZZ[\tilde{\sigma_I}]),
\end{equation}
where $\delta$ is the degree of the map $\mathbf{f}$. If the rank
drops, we assign weight zero to the semigroup $[\tilde{\sigma_I}]$.
The realization of $\Delta(\partial\overline{X})$ in $\Lambda_n$ is
the collection of the weighted semigroups $\{[\tilde{\sigma_{I}}]:
|I|=d\}$.
%
%

\begin{theorem}\label{thm:ImplicitizationIsMonomial}
  Let $\mathbf{f}\colon \TP^d\dashrightarrow \TP^n$ be a rational
  generically finite Laurent polynomial map and let $Y$ be the Zariski
  closure of the image of $\mathbf{f}$. Denote by $X\subset \TP^d$ the
  domain of $\mathbf{f}$ and let
  $(\overline{X}, \partial\overline{X})$ be a cnc pair with associated
  boundary complex $\Delta(\partial\overline{X})$.  Then, the tropical
  variety $\cT Y$ is the weighted cone over the realization of this
  complex in $\RR\otimes \Lambda_n$.
\end{theorem}

\begin{proof}
  We now justify why we can compute $\cT X'\subset \TP^{d+n}$ via
  finding a cnc pair for the \emph{open} subset $X$ of $\TP^d$. We
  build $\overline{X}$ in two steps. First, we add the boundary
  divisors $F_1, \ldots, F_n$ of $\TP^d$ given by the equations $f_1,
  \ldots, f_n$. Then, we embed $\TP^d$ inside a projective toric
  variety associated to the fan $\cT X$ and we compactify $X$ inside
  this toric variety. By~\cite[Theorem 1.2]{CompactificationsTori},
  the outcome is a cnc pair
  $(\overline{X}, \partial\overline{X})$. The components of the
  boundary $\partial \overline{X}$ come in two flavors: the divisors
  $\overline{F_j}$ obtained as the closure of $F_j$ in $\overline{X}$
  and the divisors $D_1, \ldots, D_m$ in $\overline{X}\smallsetminus
  \TP^d$. Since $X'$ is isomorphic to $X$, the cnc pair
  $(\overline{X}, \partial \overline{X})$ is also associated to
  $X'$. Notice that any choice of a cnc pair as this property. We
  choose a tropical compactification since the realization of the
  boundary complex is very explicit.

  Next, we discuss out to realize the boundary complex
  $\Delta(\partial\overline{X})$ in $\Lambda_{d+n}$.  For simplicity,
  we fix a basis $\{\chi_1, \ldots \chi_d, \zeta_1, \ldots, \zeta_n\}$
  of characters of the torus $\TP^{d+n}$ by combining bases of
  characters of $\TP^d$ and $\TP^n$.  Since $\chi_i$ is a unit in
  $\TP^d$ and $\overline{F_i}\cap \TP^d \subset \TP^d$ is locally
  defined by $f_i(\mathbf{x})$, we have
  $\val_{\overline{F_j}}(\chi_i)=0$, whereas
  $\val_{\overline{F_j}}(\zeta_i) =
  \val_{\overline{F_j}}(f_i)=\delta_{ij}$.  Similarly,
  $\val_{D_{j}}(\zeta_i)=\val_{D_{j}}(f_i)$ for all $j$.  Applying the
  projection $\alpha\colon \RR^{d+n}\to \RR^n$ to the last $n$
  coordinates from~\eqref{SIeq:14bis}, we see that each maximal cell
  $\sigma$ in $\Delta(\partial\overline{X})$ satisfies
  $\alpha([\sigma])=[\tilde{\sigma}]$.  The transition from $\cT X'$
  to $\cT Y$ is obtained by applying the linear map
  $(0|\operatorname{Id}_{n})$ and noticing that
  \[
\operatorname{index}(\alpha([\sigma])^{\sat},  \alpha([\sigma]))=  \operatorname{index}(\alpha([\sigma])^{\sat},
  \alpha([\sigma]^{\sat}))\,  \operatorname{index}([\sigma]^{\sat}, [\sigma]),
\] unless the dimension of the vector space spanned by
  $\alpha([\sigma])$ is less than $d$. Such cones do not contribute to
  the multiplicity of regular points in $\cT Y$.

We end by discussing the multiplicities on $\cT Y$.
By construction, $\delta$ equals the degree of the monomial map
$\alpha$ restricted to the variety $X'$. The push-forward formula of
multiplicities implies the transition from~\eqref{SIeq:9} to
\eqref{SIeq:21} and in particular, the addition of the factor $1/\delta$
and the replacement of the lattice index factor in $\Lambda_{n+d}$ by the
corresponding lattice index factor in $\Lambda_n$.
\end{proof}

  It is in this sense that the boundary complex $\Delta(\partial
  \overline{X})$ is ``pushed-forward'' via the map $\mathbf{f}\colon
  X\to Y$ to give the boundary complex of a cnc pair associated to
  $Y$. The key fact in the proof of this result is that $\mathbf{f}$
  induces a map on function fields $\mathbf{f}^{\#}\colon
  \CC(Y)\hookrightarrow \CC(X)$. Since the field $\CC(X)$ is a finite
  extension of $\CC(Y)$ of degree $\delta$, we can always extend any
  discrete valuation on $\CC(Y)$ to a discrete valuation on $\CC(X)$
  via the map $\mathbf{f}^{\#}$. Likewise, valuations on $\CC(X)$ can
  be restricted to $\CC(Y)$.  The realization of each vertex $v_i$ in
  $\Delta(\partial\overline{X})$ by the lattice point $[\tilde{D_k}]$
  corresponds to the image of the realization of $D_k$ in
  $\Lambda_{n+d}$ under the linear map associated to the projection
  $\alpha$ from~\eqref{SIeq:14bis}. This highlights the deep
  connections between tropical implicitization and homomorphisms of
  tori.

\section{Tropical implicitization for generic surfaces}
\label{SIsec:trop-elim-gener}

In this section, we specialize the constructions of
Section~\ref{sec:tropical-elimination} to the case of generic rational
surfaces parameterized by polynomials with fixed support. Our methods
are based on~\cite{NPofImplicitEquation}. Unlike the case of
\cite[Theorem 4.1]{NPofImplicitEquation}, our construction is
independent on the smoothness on the ambient toric variety associated
to a fan structure on the tropical variety. In addition,
we give precise certificates for the genericity of these
surfaces. 

We keep the notation from Section~\ref{sec:tropical-elimination}. Our
surface $Y\subset \TP^n$ ($n\geq 3$) is parameterized by the
generically finite Laurent polynomial map $\mathbf{f}=(f_1, \ldots,
f_n)\colon \TP^2\dashrightarrow \TP^n$. Our goal is to compute the
tropical surface $\cT Y$. To simplify the exposition, we fix a basis
of the character lattice $\Lambda_n^{\vee}$, which allows us to
identify $\Lambda_n$ with
$\ZZ^n$. Following~\cite{NPofImplicitEquation}, we assume each
coordinate of $\mathbf{f}$ is generic \emph{relative to its
  support}. That is, we fix the $n$ Newton polytopes $\cP_1, \ldots,
\cP_n$ of our polynomials $f_1, \ldots, f_n$ and we let their
coefficients vary generically. These $n$ polynomials determine $n$
curves in $\TP^2$ with equations $(f_i=0)$. Our two main players in
this section are the complement of this curve arrangement, which we
call $X$, and the fan $\NF$ obtained as the common refinement of the
$n$ inner normal fans of the polytopes $\cP_1, \ldots, \cP_n$.  After
compactifying $X$ inside the toric variety $X_{\NF}$, the genericity
condition guarantees that $(\overline{X},\partial \overline{X})$ is a
cnc pair. The combinatorial nature of $X_{\NF}$ makes it suitable for
studying generic surfaces in the moduli space associated to the map
$\mathbf{f}$.

We now state the main result in this section. The remainder will be
devoted to its proof and to give several numerical examples. For
simplicity, we assume that our choices of coefficients give distinct,
\emph{irreducible}
polynomials. 
We denote the rays of $\NF$ by $\rho_1, \ldots, \rho_m$, oriented
counterclockwise, with primitive generators $n_{\rho_1}, \ldots,
n_{\rho_m}$ in $\ZZ^2$. For each such ray $\rho\in \NF^{[1]}$, we let
$[D_{\rho}]=(\min_{\alpha \in \cP_1}\{\alpha \cdot n_{\rho}\},\ldots,
\min_{\alpha \in \cP_n}\{\alpha \cdot n_{\rho}\})$. This is precisely
the evaluation of the piecewise linear tropical map
$\trop(\mathbf{f})$ at the point $\rho$.

\begin{theorem}\label{thm:TropImplGenericSurfaces}
  The tropical variety $\cT(Y)$ is
  the cone over a weighted graph, with vertices
\[
\{e_i: \dim \cP_i\neq 0, 1\leq i \leq n\}\cup
\{[D_{\rho}]:  \rho \in \NF^{[1]}, [D_{\rho}]\neq 0\},
\]
and positively weighted edges
\begin{enumerate}
\item $m_{([D_{\rho_j}],
    [D_{\rho_k}])}=\delta^{-1} 
  {| \gcd\big( 2\times2-\operatorname{minors }\, ([D_{\rho_j}] \mid
    [D_{\rho_k}])|
  } / { |\det (n_{\rho_k} \mid n_{\rho_j} )| }$, if $|j-k|=1$ mod $m$
  or $0$ otherwise.
\item $m_{(e_i, [D_{\rho}])}=  \delta^{-1}
(|\operatorname{face}_{n_{\rho}}(\cP_i) \cap
  \ZZ^2| -1)\, \gcd\big( [D_{\rho}]_j: j \neq i\big)$, if $n_{\rho}
  \in \cT(f_i)$, or $0$ otherwise.
\item $m_{(e_i, e_j)} = \delta^{-1}
  \operatorname{length}( (f_i=f_j=0)\cap \TP^2)$ if $\dim(\cP_i + \cP_j)=2$,
  and $0$ otherwise. Under further genericity, this number
  equals $1/\delta$ times the mixed volume of $\cP_i$ and $\cP_j$.
\end{enumerate}
\end{theorem}

It is important to point out that the previous algorithm was already
presented in~\cite{NPofImplicitEquation} and further studied
in~\cite{ElimTheory}. We contribute to the subject by elucidating the
right genericity condition to impose. The proof of~\cite[Theorem
2.1]{NPofImplicitEquation} requires the genericity of both the
coefficients and the Newton polytopes, to ensure that the Minkowski
sum of the $n$ polytopes $\cP_1, \ldots, \cP_n$ is a \emph{smooth}
polytope. Our proof discards this extra assumption on the polytopes,
unraveling the key aspects in their argumentation, and extends the
result to polynomial maps with arbitrary finite degree, as in
\cite[Theorem 5.1]{ElimTheory}.
\begin{proof}
  We follow the strategy of \cite[Theorems 2.1 and
  4.1]{NPofImplicitEquation} and make the appropriate adjustments
  along the way. Our main tool will be
  Theorem~\ref{thm:ImplicitizationIsMonomial}.  We fix the arrangement
  complement $X= \TP^2 \smallsetminus \bigcup_{i=1}^n(f_i=0)$ and
  embed it in the {normal} toric surface $X_{\NF}$.  
  The compactification of $X$ induces the pair 
  $(X_{\NF}, \partial X_{\NF})$, where
\[
\partial X_{\NF}=\{\overline{F_1}, \ldots, \overline{F_n}\}
\;\bigcup\;\{D_{1}, \ldots, D_{m}\}.
\]
Here, $D_i$ denotes the toric divisor $D_{\rho_i}$ and
$\overline{F_i}$ is the divisor associated to the curve $F_j:=(f_j=0)$
in $\TP^2$ as in the proof of
Theorem~\ref{thm:ImplicitizationIsMonomial}.

The boundary $\partial X_{\NF}$ consists of two types of irreducible
components. The first class compounds the toric divisors indexed by
the rays of $\mathscr{N}$. They correspond to facets of the Minkowski
sum $\sum_{i=1}^n \cP_i$. Since the fan $\NF$ is simplicial, the toric
boundary 
is a combinatorial normal crossings divisor.  The remaining components
are the $n$ divisors $\overline{F_1}, \ldots, \overline{F_n}$,
obtained from the curves $(f_i=0)$.  The irreducibility and genericity
of the polynomials $f_i$, together with Bertini's theorem, show that
these divisors are smooth and that $(X_{\NF},\partial X_{\NF})$ is a
cnc pair. Notice that if $f_j$ consists of a single monomial, then
$\overline{F_j}$ is the empty set. Such indices do not induce a
vertex in the boundary complex $\Delta(\partial{X_{\NF}})$, so from
now on we may assume $\dim \cP_i>0$ for all $i=1, \ldots, n$.

We now analyze the combinatorial information coming from the cnc
pair. The boundary complex
$\Delta(\partial X_{\NF})$ is a graph with $m+n$ vertices. Its edges
consist of pairs of vertices in $I \cup J$, where $I\subset \{1,
\ldots, m\}$, $J\subset \{1, \ldots, n\}$. 
The first type of edges are of the form $(D_{\rho}, D_{\rho'})$ for
$\rho$ and $\rho'$ rays in the fan $\NF$.  By standard intersection
theory on 
toric varieties, we know that the intersection
numbers among the torus-invariant divisors are given by the following
formula
  \begin{equation}
    D_{\rho}\cdot D_{\rho'} =
    \begin{cases}
      1 & \;\operatorname{ if }\,\rho\, \operatorname{ and }\, \rho'\,\operatorname{ \text{generate a
        two-dimensional cone in }}\NF,\\
      0 & \;\operatorname{else}.
    \end{cases}\label{eq:18}
  \end{equation}
  This says that we only have edges among consecutive rays of $\NF$,
  and their weight is 1.

  When $|J|=1$, we seek to identify edges of the form
  $(\overline{F_j}, D_{\rho})$, for $\rho\in \NF$ and $j=1, \ldots,
  n$. Again, this is done by toric methods. Since $\overline{F_j}$
  represents a Cartier divisor with local equation $f_j$, 
  the weight of this edge is the intersection number of the initial
  form $\init_{{\rho}}(f_j)$ and $D_{\rho}$. This quantity agrees with
  the number of nonzero solutions of the univariate polynomial
  $\init_{{\rho}}(f_j)$, namely, the lattice length of the face of
  $\cP_j$ associated to the ray $\rho$. If this face is a vertex, the
  initial form is a monomial, and so the intersection number is
  zero. Thus, we see that $\overline{F_j}$ is adjacent to a node
  $D_{\rho}$ if and only if ${\rho}$ is a ray in the normal fan of
  $\cP_j$, and if so,
  \begin{equation}
    \label{eq:19}
    \overline{F_j} \cdot D_{\rho} = \operatorname{\text{ lattice length of }}
    \operatorname{face}_{{\rho}}(\cP_j) =
    |\operatorname{face}_{{\rho}}(\cP_j) \cap \ZZ^2| -1.
  \end{equation}

  Finally, if $|J|=2$, we want to certify which edges
  $(\overline{F_i}, \overline{F_j})$ belong to the boundary complex
  $\Delta(\partial X_{\NF})$. We claim it suffices to check if the
  equations $f_i$ and $f_j$ have a common root in $\TP^2$ since any
  remaining intersection points would lie in the toric boundary, thus
  contradicting the cnc property of the chosen pair.  Therefore, the
  weight of this edge is the length of the zero-dimensional scheme
  $(f_i=f_j=0)\cap \TP^2$.  If the coefficients of these polynomials
  are generic enough, Bernstein's theorem implies that this number is the
  mixed volume of the polytopes $\cP_i$ and
  $\cP_j$. 
  The mixed volume is nonzero if and only if the Minkowski sum of the
  corresponding polytopes is two-dimensional. This explains the extra
  assumption $\dim ( \cP_i+ \cP_j) =2$ in the statement.  Notice that
  since we are interested in the weighted boundary complex,
  we can safely assume that the dimension restriction characterizes
  the edges $(\overline{F_i}, \overline{F_j})$.  Artificial edges
  added to the boundary complex have weight zero.

\smallskip

It remains to discuss the realization of the boundary complex in
$\RR^n$.  By Theorem~\ref{thm:ImplicitizationIsMonomial}, we know that
$\val_{\overline{F_j}}(f_i)=\delta_{i,j}$. 
We compute the divisorial valuation of all $D_\rho$'s with the tools
of toric geometry~\cite[Section 5.2]{FultonToricBook}. Without loss of
generality, we may assume $n_\rho=e_1$. By definition,
$\val_{D_{\rho}}(f_j)$ is the order of vanishing of the polynomial
$f_j$ at $D_\rho$, that is, by the maximal exponent of $x_1$ dividing
$f_j$ in the polynomial ring $\CC[x_2^{\pm 1}][x_1]$. Notice that this
number can be negative.  The maximum exponent is precisely
$\trop(f_j)(e_1):=\min_{\alpha \in \cP_j}\{ e_1\cdot \alpha\}
$. We infer,
  \begin{equation*}
    \label{eq:17}
    [D_i]:=(\val_{D_i}(f_j))_{j=1}^n = 
    (\trop(f_1)(n_{\rho_i}),  \ldots,   \trop(f_n)(n_{\rho_i}))
    =\trop(\mathbf{f})(n_{\rho_i}) \quad \forall\,
    i=1, \ldots, m.
  \end{equation*}
  Theorem~\ref{thm:ImplicitizationIsMonomial},
  expressions~\eqref{eq:18} and~\eqref{eq:19} yield the desired
  multiplicities.
\end{proof}
\begin{example} Our first example is a modification of~\cite[Example
  3.4]{NPofImplicitEquation}, where we remove a monomial factor from
  each
  polynomial. 
  This change has no effect on the combinatorics of the
  graph, but distorts its realization and the corresponding implicit
  equation. Our general surface $Y\subset \TP^3$ is parameterized by
\begin{equation*}
   f_1(s,t)  = a_1+a_2\,s^2t+a_3\, st^2,\qquad
    f_2(s,t)  = b_1\, st + b_2\, s+ b_3\,t,\qquad
    f_3(s,t)  = c_1\, t + c_2\, s^2 + c_3\, st^2,
  \end{equation*}
where $a_1,a_2,a_3, b_1, b_2, b_3, c_1, c_2, c_3\in \CC$ are generic
nonzero coefficients. The map has degree $\delta=1$. The non-smooth
fan $\NF$ has nine rays but they yield only eight vertices in the
realization of $\Delta(\partial \overline{X})$: $ [D_{1}]=[D_4]=
(-2,-1,-2), [D_{2}]= (-5,-3,-4), [D_{3}]= (-3,-2,-3), [D_{5}]=(-1, -1,
-1), [D_{6}]=(0, -1, -1), [D_{7}]=(0, 1, 1), [D_{8}]=(0, 1, 2),
[D_{9}]=(0, -1, -2)$.
Likewise, the realization of the edges $(D_6, D_7)$ and $(D_8, D_9)$
in $\Delta(\partial X_{\NF})$ give one-dimensional cones in $\cT
Y$. We indicate this by drawing a dashed edge in the abstract graph.
The weights of all 19 edges are computed using mixed volumes, and are
indicated in the left of Figure~\ref{fig:NPDegenerateExample}.

The resulting weighted graph in $\RR^3$ has four bivalent vertices (in gray) and
it is depicted on the right of
Figure~\ref{fig:NPDegenerateExample}. After removing these gray
vertices, we obtain a graph with $f$-vector $(7,13)$.  The complement
of the graph has eight connected components. Notice that the vertices
$e_2, [D_1]=[D_4], [D_3]$ and $[D_5]$ are aligned in the picture since
they generate a two-dimensional cone in $\RR^3$. In addition to the
four bivalent vertices, this also explains the difference between the
number of edges in the boundary complex and its realization. The predicted
edge $([D_4], [D_5])$ 
can be seen as the arc containing the vertices $[D_3],
[D_4]$ and $[D_5]$.
   \begin{figure}[htb]
     \begin{minipage}[c]{.3\linewidth}
       \includegraphics[scale=0.4]{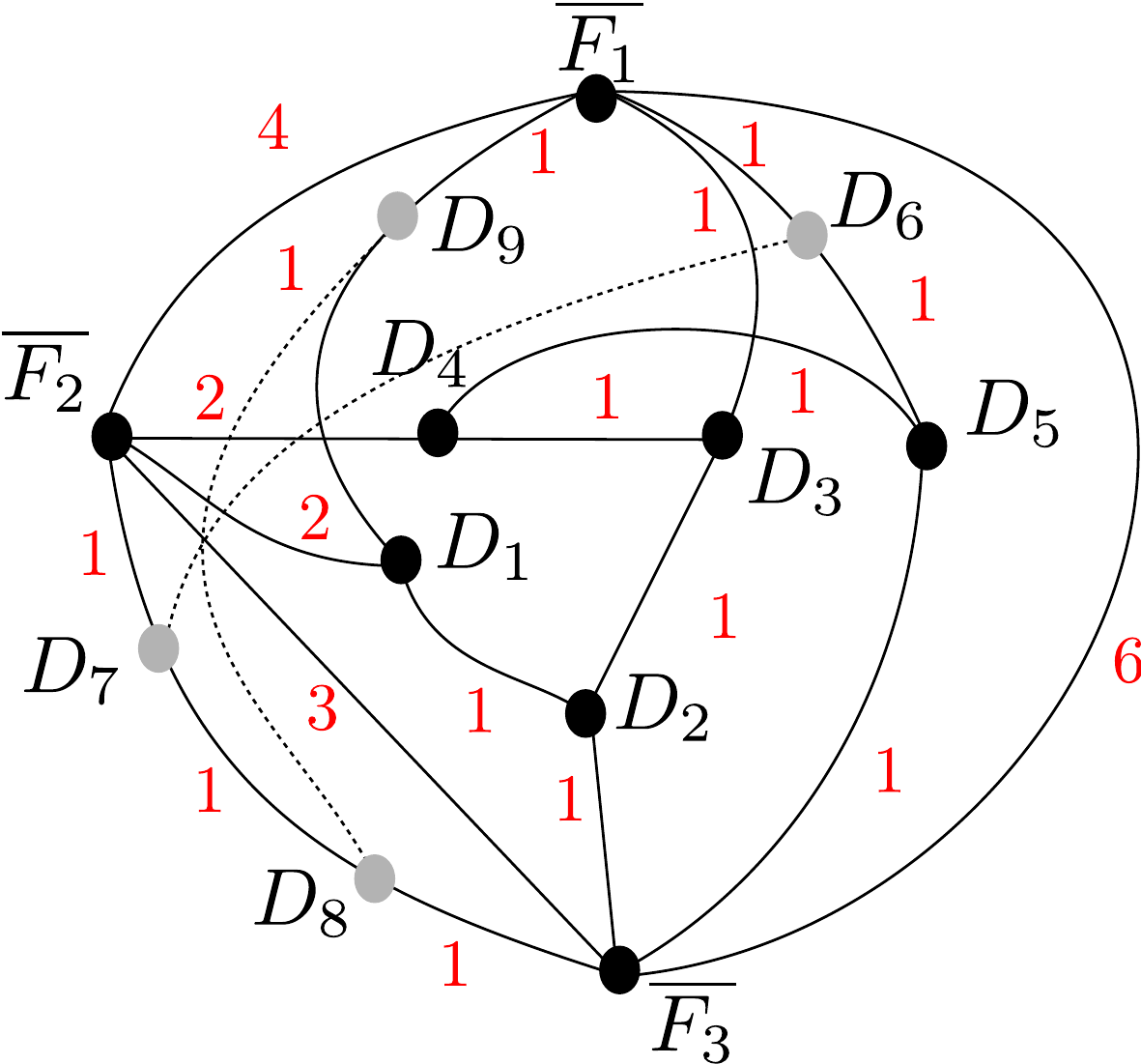}
     \end{minipage}
     \begin{minipage}[c]{.1\linewidth}
\hfill       \huge{$\rightsquigarrow$} \hfill
     \end{minipage}
     \begin{minipage}[c]{.4\linewidth}
       \includegraphics[scale=0.4]{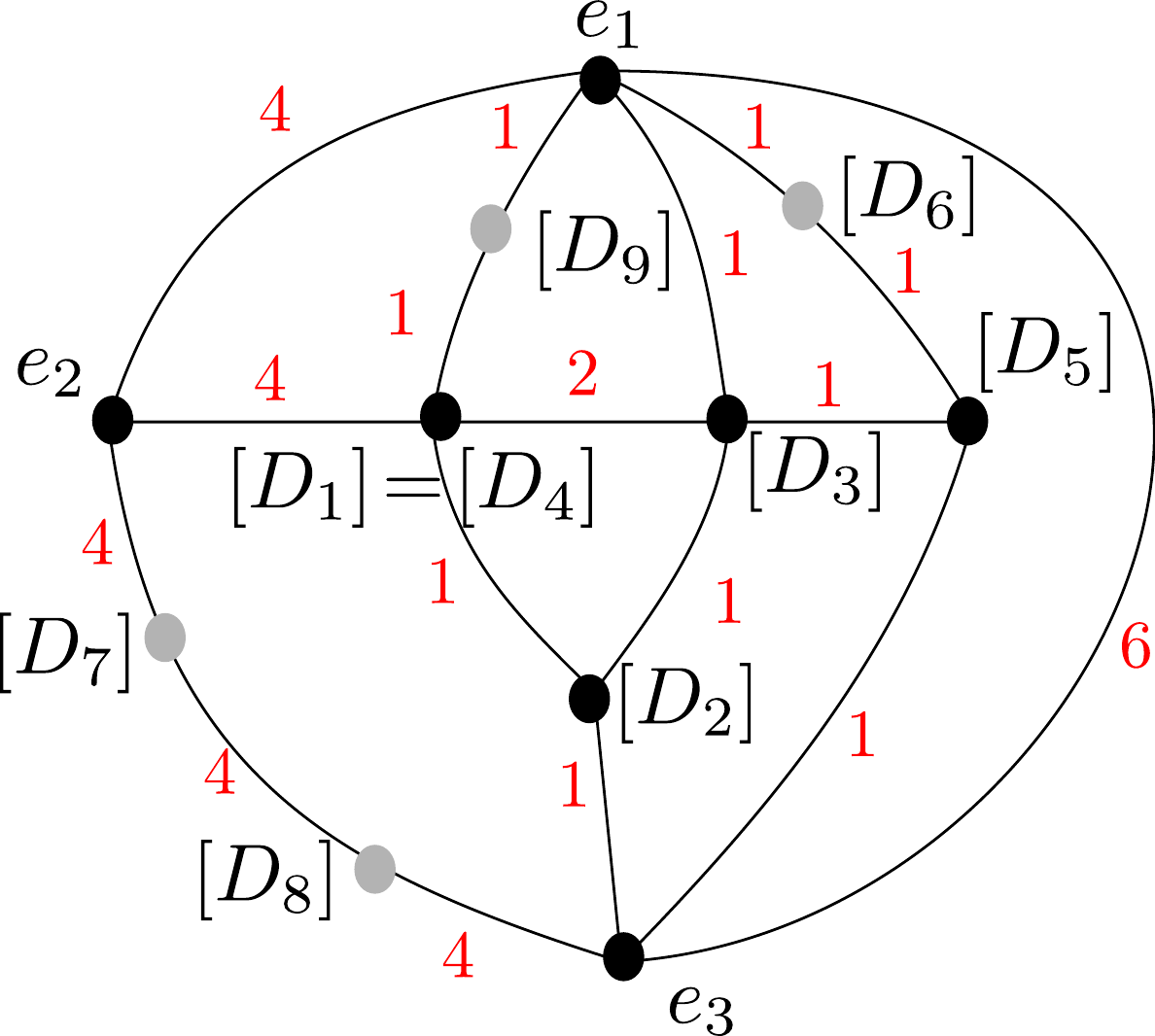}
     \end{minipage}
     \caption[Tropical graph of a generic surface in $\TP^3$.]{From
       left to right: weighted graphs representing $\cT Y$. The left
       one corresponds to the abstract graph and the right one is the
       planar graph obtained by realizing the abstract graph and
       combining weights of overlapping edges. The dashed edges on the
       left graph have weight zero and they disappear in the planar
       graph.}
     \label{fig:NPDegenerateExample}
   \end{figure}

   For generic choices of coefficients $a_1, \ldots, c_3$, the
   implicit polynomial has degree 14~\cite{singular}. Its Newton
   polytope has $f$-vector $(8,13, 7)$, which matches the
   combinatorics of our graph.
%
%
%
%
%
\end{example}

\begin{example}\label{SIex:GenericAlpha}
  We consider the morphism $f=(f_1, f_2, f_3)\colon \afc^2\dashrightarrow Y\subset\afc^3$
given by 
\begin{equation*} \label{SIeq:4}
   f_1(s,t) = a_1\,s^2+a_2\,s^3 +a_3\,t^2, \;
    f_2(s,t)  = b_1\,t^2+b_2\,t^3+b_3\,s^2,\;
    f_3(s,t)  = c_1\,st +c_2\,s^3+c_3\,t^3+c_4\,st^2+c_5\,s^2t,
\end{equation*}
with generic coefficients $a_1, \ldots, c_5\in \CC^*$.  The map has
degree one and the normal fan $\NF$ has eight rays, three of which
have non-trivial weights 2, 2 and 3.

The vertices of the graph have coordinates $e_1, e_2, e_3$, $[D_1]=
(0,0,0)$, $[D_2]= (-9,-6, -9)$, $[D_3]=(-3,-3,-3)$,
$[D_4]=(-6,-9,-9)$, $[D_5]=(0,0,0)$, $[D_6]=(2,2,3)$, $[D_6]=(2,2,3)$,
$[D_7]=(2,2,2)$ and $[D_8]=(2,2,3)$. After going through dimension
testings, we obtain a list of fourteen edges as seen in the right of
Figure~\ref{fig:GenericAlphaCurveExample}, whose weights we can compute
via mixed volumes. The transition from the weighted abstract graph to
its realization is seen in Figure~\ref{fig:GenericAlphaCurveExample}.
\end{example}
   \begin{figure}[htb]
     \begin{minipage}[c]{.35\linewidth}
    \includegraphics[scale=0.4]{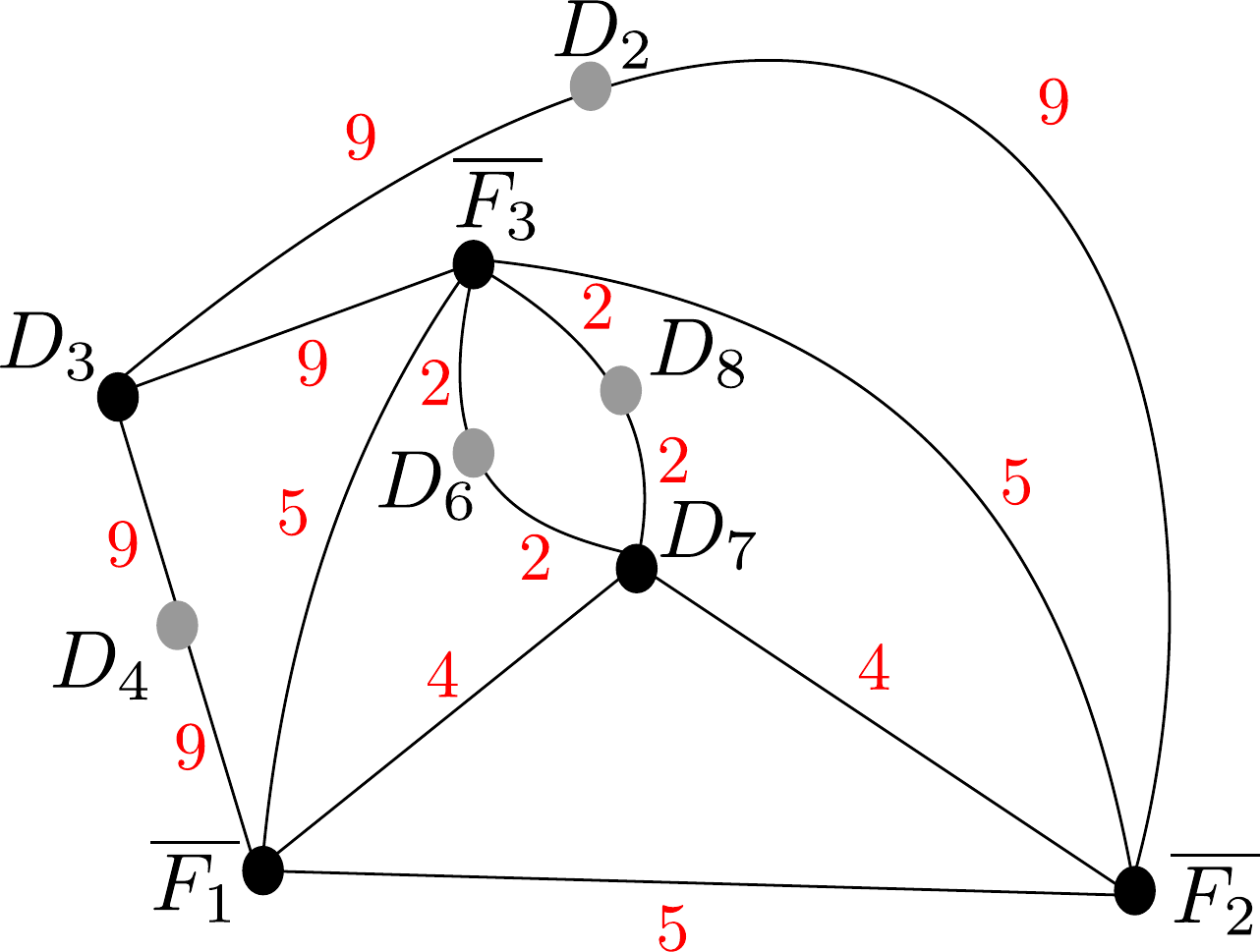}
     \end{minipage}
     \begin{minipage}[c]{.1\linewidth}
       \hfill\huge{$\rightsquigarrow$} \hfill
     \end{minipage}
     \begin{minipage}[c]{.3\linewidth}
    \includegraphics[scale=0.4]{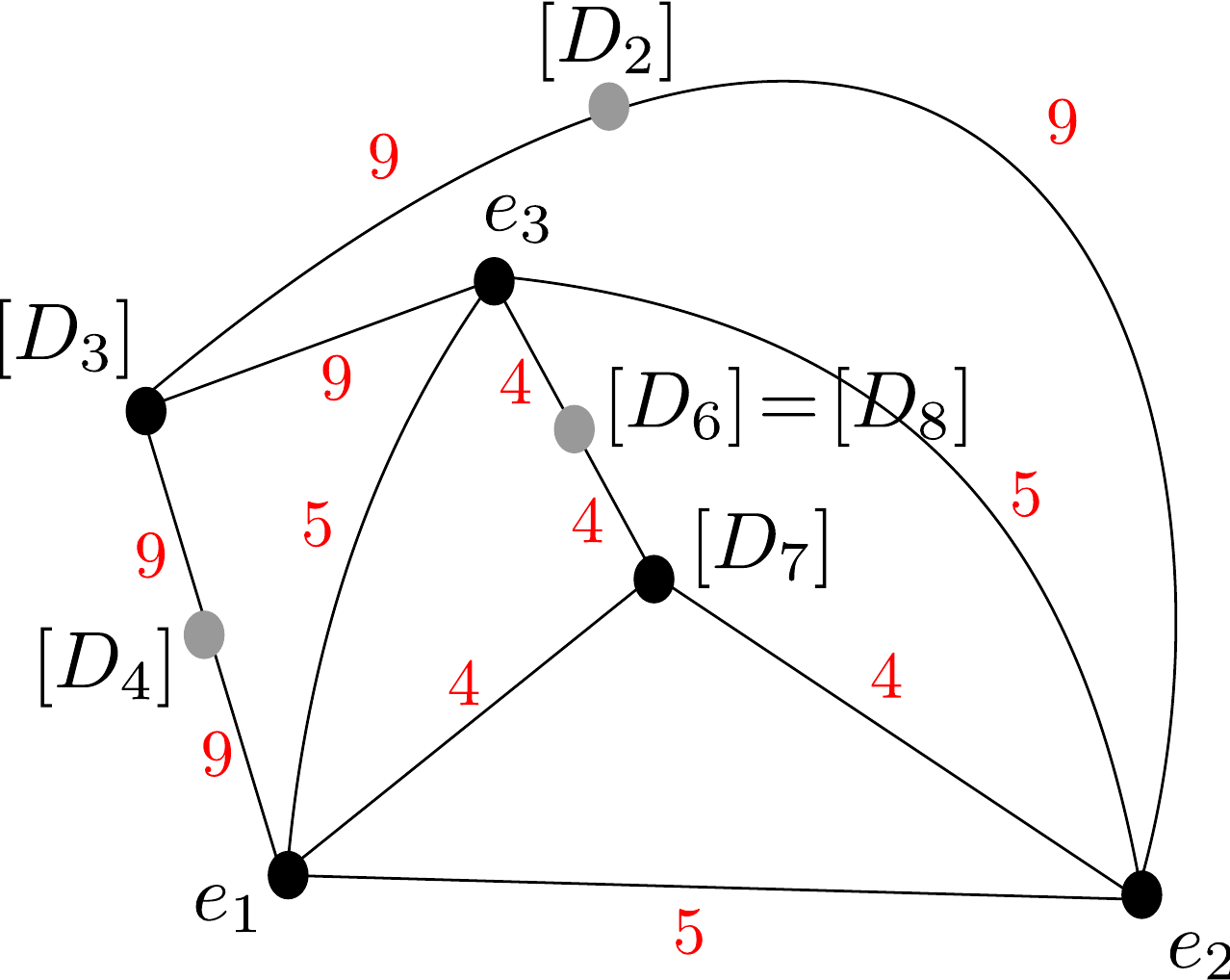}
     \end{minipage}
     \caption[Tropical graph of a generic surface associated to three
     generic nodal curves.]{Weighted graphs representing $\cT Y$.}
     \label{fig:GenericAlphaCurveExample}
   \end{figure}

\begin{example}\label{SIex:BerndsExample}
  As our third example we consider the surface in $\CC^3$
  parameterized by the degree one morphism $f=(f_1, f_2, f_3)\colon
  \afc^2\dashrightarrow Y$, where
\begin{equation} \label{SIeq:5}
     f_1(s,t)  = a_1+a_2\,s +a_3\,t,\qquad
    f_2(s,t)  = b_1+b_2\,t+b_3\,s^2,\qquad
    f_3(s,t)  = c_1 + c_2\, st.
\end{equation}
Using the methods described in this section we obtain a weighted graph
with seven vertices 
$e_1, e_2, e_3$, $[D_2]= (-1,-2,0)$, $[D_3]=(-1,-2,-2)$,
$[D_4]=(-2,-2,-3)$ and $[D_5]=(-1,-1,0)$.  After removing the bivalent
vertices $[D_2]$ and $[D_5]$, we get a graph with $f$-vector $(5,8)$,
whose complement has five connected components. The eight edges are
$(e_1,e_2)$, $(e_1, e_3)$ (both with weight 2), $(e_2, e_3)$ (with
weight 3), $(e_1, [D_3])$ (with weight 2), and $(e_2,
[D_4])$, $(e_3, [D_3])$, $(e_3,
[D_4])$ and $([D_3],[D_4])$ (all with weight 1). Its support can be
obtained from the rightmost picture in
Figure~\ref{SIfig:BerndNonGenericExample} by removing the vertex
$[E_3]$ and its three adjacent edges.

On the other hand, by standard elimination techniques, we see that the
implicit equations is a dense polynomial of degree 3 in $x,y,z$ with five
extreme monomials $1, x^3, x^2y, y^2$ and $z^2$. Its coefficients are
polynomials in the indeterminates $a_1$ through $c_2$. 
%
In Section~\ref{SIsec:trop-elim-non} we revisit this example and
explain how certain specializations of the coefficients $a_1$ through
$c_2$ removes the extremal monomial $1$ and hence gives a new facet to
the polytope. This choice of coefficients destroys the genericity
conditions on the polynomial map $\mathbf{f}$.
\end{example}


\section{Tropical implicitization for non-generic surfaces}
\label{SIsec:trop-elim-non}

In this section, we discuss methods for computing the tropicalization
of non-generic parametric surfaces.  As in
Section~\ref{SIsec:trop-elim-gener}, we start from a generically
finite Laurent polynomial map $\mathbf{f}=(f_1, \ldots, f_n)\colon
\TP^2\dashrightarrow \TP^n$.  We assume that the polynomials have
fixed support and we allow special choices of coefficients that
preserve their Newton polytopes by such that $(X_{\NF}, \partial
X_{\NF})$ is not a cnc pair.  We explain how to solve this issue and
present numerical examples that illustrate the algebro-geometry
complexity of the problem.

As we discussed in the generic case, we aim to find a cnc pair
$(\overline{X}, \partial \overline{X})$ associated to the arrangement
of plane curves $X=\TP^2 \smallsetminus\bigcup_{i=1}^n (f_i=0)$.  The
following lemma implies that we can assume all $f_i$'s are
\emph{irreducible}. A similar result allows us to assume all
polynomials are distinct.
\begin{lemma}
  Assume $\mathbf{f}$ is a finite map and that $f_1$ factors as
  $f_1=gh$ with $\deg g, \deg h<\deg f_1$. Then, the map
  $\mathbf{f'}=(g,h, f_2, \ldots, f_n)\colon X \to \TP^{n+1}$ is
  generically finite and $\mathbf{f}=\beta\circ \mathbf{f'}$, where
  $\beta\colon \TP^{n+1}\to \TP^n$ sends $(t_0, t_1, \ldots, t_n)$ to
  $(t_0t_1, t_2, \ldots, t_n)$. In addition, $\beta$ restricted to
  the image of $\mathbf{f'}$ is generically finite.
\end{lemma}
As a first attempt to answer our question, we apply generic methods
from Section~\ref{SIsec:trop-elim-gener} and compactify $X$ via its
embedding in the projective toric variety $X_{\NF}$.  
The non-genericity of the coefficients of $\mathbf{f}$ says precisely that 
$(X_{\NF}, \partial X_{\NF} )$
is not a cnc pair.  Since the excessive intersection points need not
be torus invariant (and will not be in general), toric blow-ups cannot
be used to achieve the desired condition. Instead, we can resolve
toric singularities on the ambient space $X_{\NF}$ by toric blow-ups,
refining $\NF$ to a smooth fan $\NF'$ in $\RR^2$, perform classical
point blow-ups on the smooth surface $X_{\NF'}$ and finally pull back
the boundary divisors along this resolution. This procedure is tedious
to do in practice.  Our alternative strategy 
does not take advantage of the combinatorial input data, yet it is
simpler to carry out in explicit calculations.

Given $\mathbf{f}$ and $X$ as above, we consider its compactification
$\overline{X}$ in $\pr^2$. This set has $n+1$ boundary divisors:
$F_i=(f_i=0)$ and $F_{\infty}=(x_3=0)$. Let $\pi\colon \tilde{X}\to
\overline{X}$ be any resolution of $\overline{X}$ obtained by blowing
up all intersection points of three or more boundary components (if
they exist), so that $(\tilde{X}, \partial \tilde{X})$ is a cnc pair. Let
$E_{1}, \ldots, E_{s}$ be the corresponding exceptional divisors and
$F_{\infty}', F_{i}'$ be the strict transforms of the divisors
$F_{\infty}, F_{i}$, $i=1, \ldots, n$. We write
\[
\pi^*(F_{\infty})= F_{\infty}' + \sum_{j=1}^s b_j E_j,
 \qquad
\pi^*(F_{i})=F_i' + \sum_{j=1}^s b_{ij} \cdot E_j, \quad i=1, \ldots, n,
\]
for suitable $b_{ij}, b_j\in \ZZ$.  We let $\Gamma$ be the realized
weighted boundary complex $\Delta(\partial \tilde{X})$ in $\RR^n$. The
vertices of $\Gamma$ are
\[
 [F_{\infty}']=(-\deg f_1, \ldots, -\deg {f_n}), \qquad [F_i']=e_i,\quad i=1,
\ldots, n,\]
 \[[E_{j}]=
(b_{1j} - b_j\deg f_1, \ldots, b_{nj} -b_j\deg f_n),\quad j=1, \ldots,
s.
\]
The weight of an edge $(v,w)$ equals
\[
m_{(v,w)}=\frac{1}{\delta} \, i(v,w)\, \gcd(2\times 2-\operatorname{minors
}(v\mid w)),
\]
where $i(v,w)$ is the intersection number of the associated boundary
divisors. An edge $(v,w)$ belongs to $\Gamma$ if it has positive
weight. We conclude:
\begin{theorem}\label{thm:Non-GenericTropSurfaceImplicitization}
  The tropical surface associated to the image of the 
map $\mathbf{f}\colon \TP^2\dashrightarrow \TP^n$ is the cone
  over the weighted graph $\Gamma$.
\end{theorem}

Before discussing the proof, it is instructive to analyze the
transition from $\Delta(\partial \overline{X})$ to $\Delta(\partial
\tilde{X})$. As we know, $\Delta(\partial \overline{X})$ contains a
maximal cell $\sigma_I$ of dimension at least two. The index set $I$
corresponds to an intersection of $|I|$ boundary divisors. Any blow-up
in this intersection produces a subdivision of $\delta_I$ (possibly
removing boundaries), ultimately leading to a graph.  At each step of
the resolution, the excessive intersection point gives rise to an
exceptional divisor and the remaining bad crossing points have lower
multiplicity. The boundary complex $\Delta(\partial \tilde{X})$ is
obtained by gluing all these resolution diagrams along common labeled
vertices and also adding edges corresponding to pairwise
intersections. The realization of this graph in $\RR^n$ is read off
from 
the proper transforms of the components of $\partial \overline{X}$.

\begin{proof}
  As explained earlier, our starting point is the naive
  compactification of $X$ in $\pr^2$. We extend the map $\mathbf{f}$
  by a homogeneous degree zero rational function
  $\mathbf{\tilde{f}}\colon \overline{X}\to Y$. Namely,
  $\tilde{f}_i=f_i^h/x_0^{\deg f_i}$ where $f_i^h$ is the
  homogenization of $f_i$ with respect to the new variables $x_0$.
 
  The boundary $\partial \overline{X}$ has $n+1$ irreducible
  components: the $n$ divisors $F_i=(f_i^h=0)\subset \pr^2$, $i=1,
  \ldots, n$ and the divisor at infinity $F_{\infty}=(x_0=0)$.  By
  construction, the pull-back along $\mathbf{\tilde{f}}$ of the basis
  of characters $\{\chi_1, \ldots, \chi_n\}$ is
  \[
  \mathbf{\tilde{f}}^*(\chi_j) = F_j + (- \deg f_i)\cdot F_{\infty}
  \qquad j=1, \ldots, n.
  \]
  Finally, we take a resolution $\pi\colon \tilde{X} \to \overline{X}$
  by blowing up the excessive boundary intersection points.  The set
  $\tilde{X}$ together with the map $g= \mathbf{\tilde{f}} \circ \pi$
  gives us the desired cnc pair $(\tilde{X}, \partial \tilde{X})$ and
  its realized boundary complex $\Delta(\partial\tilde{X})$.  The
  result now follows from Theorem~\ref{thm:ImplicitizationIsMonomial}.
%
\end{proof}

The following two numerical examples illustrate
Theorem~\ref{thm:Non-GenericTropSurfaceImplicitization}. They
correspond to special choices of coefficients in
Examples~\ref{SIex:GenericAlpha} and~\ref{SIex:BerndsExample}.  We
show how the original boundary complexes and the induced tropical
surfaces need to be modified in order to obtain the associated
non-generic objects.  To simplify notation, we let $s,t$ be our domain
parameters and $u$ be the homogenizing variable.

\begin{example}\label{ex:AlphaNonGeneric}
  We consider a particular choice of coefficients in
  Example~\ref{SIex:GenericAlpha}. In this case, our degree one map is
  given by the following three bivariate polynomials:
\begin{equation*} \label{SIeq:1}
    f_1(s,t)  = s^2-s^3-t^2,\quad
    f_2(s,t)  = t^2-t^3-s^2,\quad
    f_3(s,t)  = 4st -s^3-t^3-3st^2-3s^2t.
\end{equation*}
%
Since our polynomials $f_1, f_2, f_3$ have nonnegative exponents, we
consider $X=\afc^2\smallsetminus \bigcup_{i=1}^3(f_i=0)$ and its
compactification in $\pr^2$. In this case, all three divisors
intersect at the origin. After four blow-ups, we obtain the cnc pair
$(\tilde{X}, \partial \tilde{X})$.

Let $g =f \circ \pi\colon \tilde{X} \to Y$ be as in the proof of
Theorem~\ref{thm:Non-GenericTropSurfaceImplicitization}. Then,
$g^*(\chi_1)=F_1+2E_1+3E_2+3E_3+4 E_4 -3 F_{\infty}$,
$g^*(\chi_2)=F_2+2E_1+3E_2+3E_3+4E_4-3F_{\infty}$, $g^*(\chi_3)=
F_3+2E_1+2E_2+2E_3+2E_4-3F_{\infty}$.  Thus, $[F_i]=e_i$,
$[F_{\infty}]=(-3,-3,-3)$, $[E_1]=(2,2,2)$, $[E_2]=[E_3]=(3,3,2)$, and
$[E_4]=(4,4,2)$.  The graph of $\cT Y$ has six vertices and twelve
edges and it is illustrated in
Figure~\ref{fig:alphaCurveExample}. Notice that the boundary complex
$\Delta(\partial \tilde{X})$ has one bivalent vertex and two vertices
$E_2$ and $E_3$ that map to the same integer vector. If we contract
the divisor $E_1$ of $\tilde{X}$ that has negative self-intersection,
we obtain a cnc pair with singularities whose boundary complex is build from
$\Delta(\partial \tilde{X})$ by removing the bivalent vertex and
merging the two edges $(F_3, E_1)$ and $(E_1,E_2)$ into a unique edge
$(F_3, E_2)$. This shows that smoothness of the cnc
pair is not required for geometric tropicalization.
   \begin{figure}[ht]
      \centering
    \includegraphics[scale=0.35]{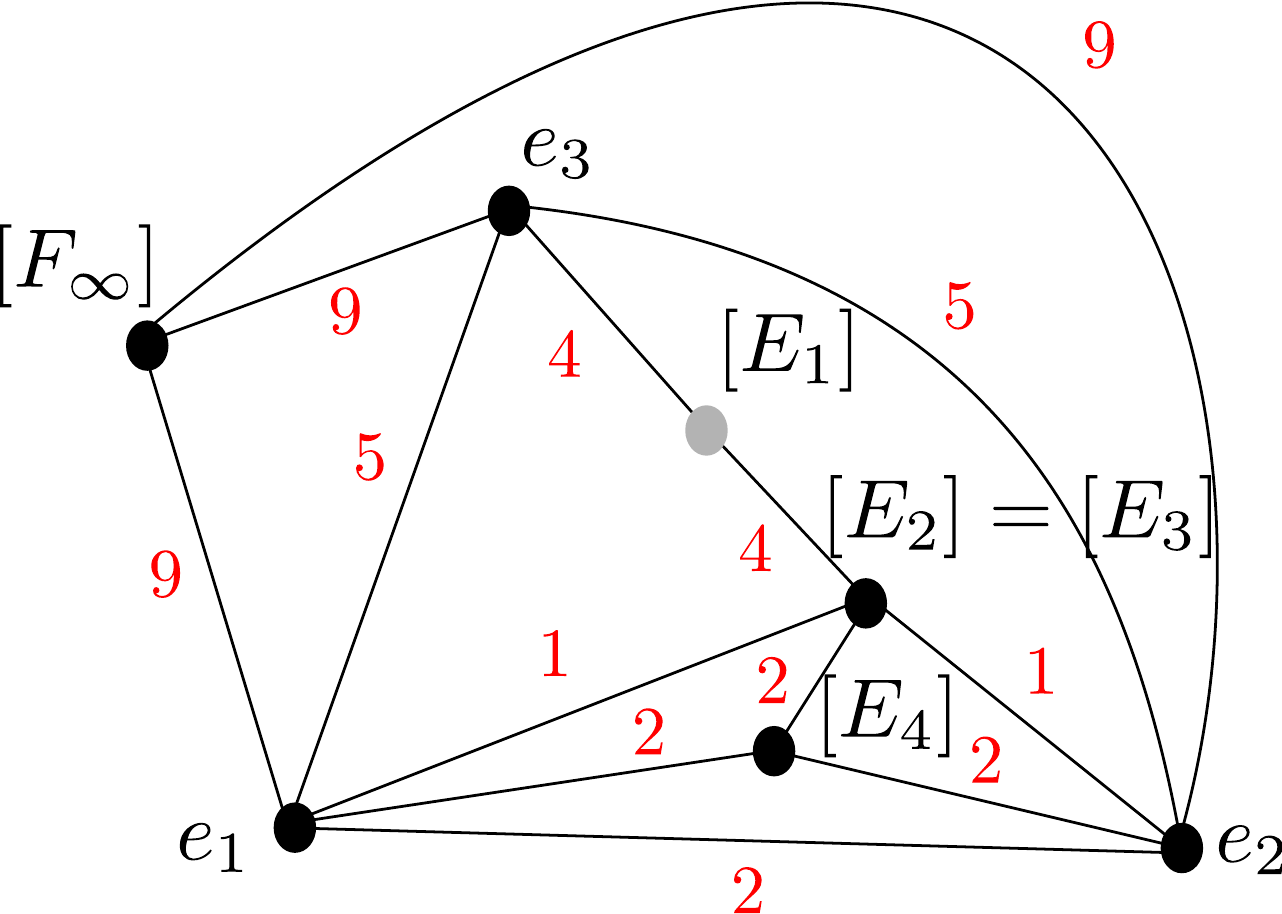}
      \caption[Tropical graph of a non-generic surface
      associated to three nodal curves.]{Weighted simplicial complex representing $\cT Y$.}
     \label{fig:alphaCurveExample}
   \end{figure}
\end{example}

\begin{example}\label{ex:BerndNonGeneric}
  We choose special parameter values for the map~\eqref{SIeq:5} in
  Example~\ref{SIex:BerndsExample}. The given non-generic surface $Y$
  in $\TP^2$ is parameterized by a degree one map:
  \begin{equation} \label{SIeq:6} 
f_1(s,t) = -1- s+t, \qquad f_2(s,t) =   -1+ t-s^2, \qquad f_3(s,t) = 2 - st.
\end{equation}
This choice of coefficients eliminates the constant term from the
implicit equation of $Y$ provided in Example~\ref{SIex:BerndsExample}
while preserving the supports of the three polynomials $f_1, f_2, f_3$.
Hence, the graph has one extra vertex, associated to the extra facet
that appears in the Newton polytope (see
Figure~\ref{SIfig:BerndNonGenericExample}).  The compactification of
$X= \CC^2\smallsetminus \bigcup_{i=1}^3 (f_i=0)$ in $\pr^2$ has two
triple intersection points: $(1:2:1)$ and
$(0:1:0)$. Figure~\ref{SIfig:BerndNonGenericExample} shows the
corresponding resolution diagrams.  The realization of the boundary
complex $\Delta(\partial \tilde{X})$ in $\RR^3$ follows from the
pullback of the basis of characters in $\TP^3$:
\begin{equation*}
\left\{\begin{aligned}\label{SIeq:7}
  (\mathbf{\tilde{f}}  \circ \pi)^*( \chi_1) &= F_1 - F_{\infty}- E_1 -2 E_2 + E_3,\\
 ( \mathbf{\tilde{f}} \circ \pi)^*( \chi_2) &= F_2 - 2F_{\infty} - E_1 -2 E_2 + E_3,\\
 ( \mathbf{\tilde{f}} \circ \pi)^*( \chi_3) &= F_3 - 2F_{\infty} - E_1 -3 E_2 + E_3.
\end{aligned}\right. 
\end{equation*}
Therefore, $[F_i]=e_i$ ($i=1, 2, 3$), $[F_{\infty}]=(-1,-2,-2)$,
$[E_1]=(-1,-1,-1)$, $[E_2]=(-2,-2,-3)$ and $[E_3]=(1,1,1)$. In
addition, the nonzero intersection multiplicities are $F_1\cdot
F_2=F_1\cdot F_3 = E_1\cdot F_3=E_2\cdot F_2 = E_2 \cdot F_{\infty}=
E_2\cdot E_3 = E_3\cdot F_i= 1$ ($i=1,2,3$) and $F_2\cdot F_3=2$. By
construction, we know that all edges have weight one, except for the
edges $(e_2,e_3)$ and $(e_1, [F_{\infty}])$, whose weight equals
two. The resulting graph and the Newton polytope of the defining
equation are shown in Figure~\ref{SIfig:BerndNonGenericExample}. 
\begin{figure}[ht]
  \centering
  \begin{minipage}{0.25\linewidth}
    \begin{minipage}[c]{.5\linewidth}
      \includegraphics[scale=0.33]{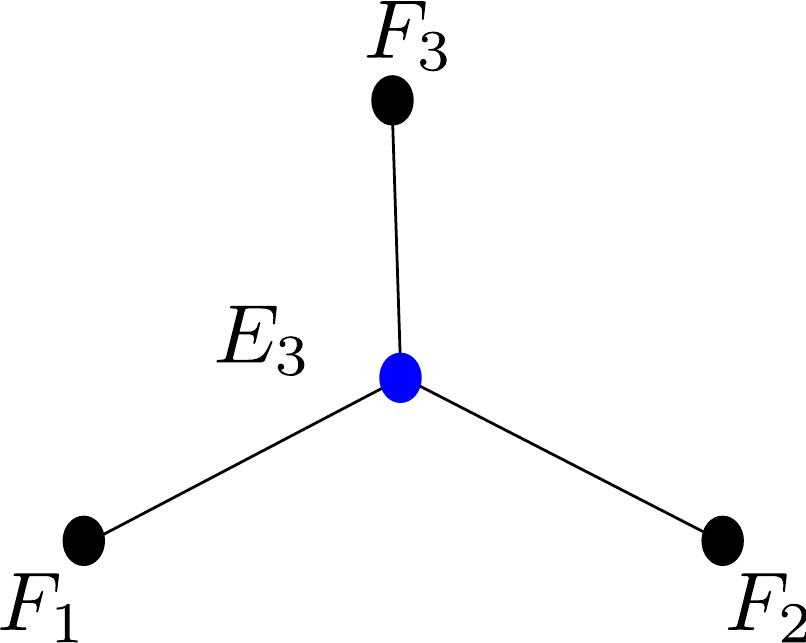}
    \end{minipage}
    \begin{minipage}[c]{.5\linewidth}
      \hspace{20ex} \includegraphics[scale=0.33]{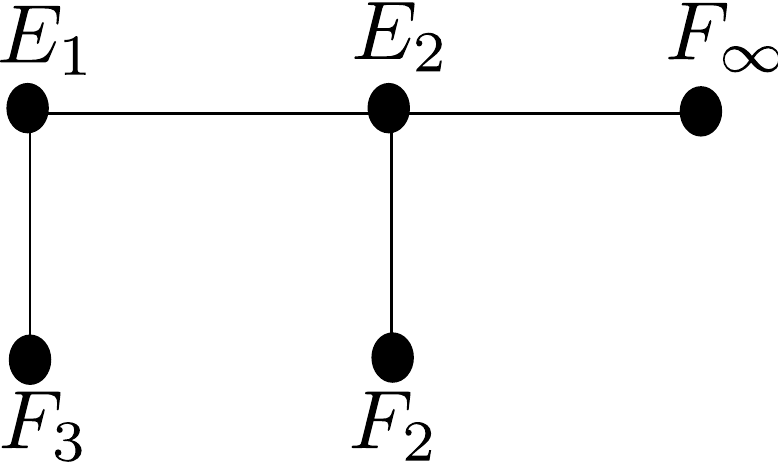}
    \end{minipage}
  \end{minipage}
  \begin{minipage}[c]{.25\linewidth}
\includegraphics[scale=0.8]{BerndExampleNoConstant.0}\hfill
\end{minipage}
 \begin{minipage}[c]{.3\linewidth}
\includegraphics[scale=0.65]
{BerndExampleNoConstant.2}
\end{minipage}
\caption[Newton polytope and dual graph of a non-generic surface in
$\afc^3$.]{From left to right: Resolution diagrams at $(1:2:1)$ and $(0:1:0)$, Newton polytope and dual graph of the non-generic surface
  in $\afc^3$ as in~\eqref{SIeq:6}.}
  \label{SIfig:BerndNonGenericExample}
\end{figure}
\end{example}

As Theorem~\ref{thm:Non-GenericTropSurfaceImplicitization} show, the
transition form the special to the generic case of tropical
implicitization of surfaces can be done at the cost of resolving
excessive intersections of plane curves. In addition to knowing the
resolution diagrams, we need to carry the intersection numbers and
divisorial valuations along the way. The examples presented show how
hard it is to predict the combinatorics of the resolution by looking
at the initial curve arrangement. The final divisorial valuations of
the exceptional divisors heavily depend on the topology of the
original plane curves.

The standard approach to obtain such valuations was introduced in work
of Enriques and Chisini~\cite{EnriqueschisiniII} and further developed
with the notions of Enriques and dual diagrams~\cite{Wall}. Such
methods are based on the topological type of the branches of the
resolved curves. Furthermore, to compute pairwise intersection numbers
of boundary divisors, we need to effectively compute this resolution,
which is difficult to carry out in concrete examples. The main
obstruction to predict these numbers without performing the resolution
lies in the construction of clusters of infinitely near points of each
singularity~\cite{Clusters}. These clusters are precisely the point
configurations emanating from successive blow-ups.

In the last years, a new object combining both Enriques and dual
graphs was introduced by Popescu-Pampu under the name of
\emph{kite}~\cite{kite}. In his language, clusters of infinitely near
points are called constellations. This kite has a natural
interpretation in the valuative tree of Favre and
Jonsson~\cite{ValuativeTree} and it seems to provide the best
framework to study arrangements of plane curves. We hope these tools
will shed some light on tropical implicitization of non-generic
surfaces.

\section*{Acknowledgments}
\label{sec:acknowledgements}
I wish to acknowledge Bernd Sturmfels for suggesting this problem to
me, and to Dustin Cartwright, Alicia Dickenstein, Eric Katz and Diane
Maclagan for fruitful conversations.  Special thanks go to Jenia
Tevelev for discussions on geometric tropicalization, which led to
Theorem~\ref{SIthm:MainTheoremMultiplicitiesSNC}.

{The author was supported by the National Science Foundation
  under the Grant DMS-0757236 (USA), by an Alexander von Humboldt
  Postdoctoral Research Fellowship (Germany) and by the Institut
  Mittag-Leffler (Sweden) as an AXA Mittag-Leffler
  postdoctoral fellow of the Spring 2011 special program ``Algebraic
  Geometry with a view towards applications.'' }

\bibliographystyle{abbrv}
\bibliography{bibliographyMerge}
\label{sec:biblio}

\bigskip

\end{document}